\documentclass{amsart}

\usepackage{amsmath}
\usepackage{amsthm}
\usepackage{amssymb}
\usepackage{color}
\usepackage{hyperref}
\usepackage{amsfonts,graphics,amsthm,amsfonts,amscd,latexsym}
\usepackage{epsfig}
\usepackage{flafter}
\usepackage{mathtools}
\usepackage{comment}
\usepackage{stmaryrd}

\usepackage{mathabx,epsfig}
\usepackage{ stmaryrd }

\hypersetup{
	colorlinks=true,    
	linkcolor=blue,          
	citecolor=blue,      
	filecolor=blue,      
	urlcolor=blue           
}
\usepackage{tikz}
\usetikzlibrary{graphs,positioning,arrows,shapes.misc,decorations.pathmorphing}

\tikzset{
	>=stealth,
	every picture/.style={thick},
	graphs/every graph/.style={empty nodes},
}

\tikzstyle{vertex}=[
draw,
circle,
fill=black,
inner sep=1pt,
minimum width=5pt,
]
\usepackage[position=top]{subfig}

\calclayout

\usetikzlibrary{decorations.pathmorphing}
\tikzstyle{printersafe}=[decoration={snake,amplitude=0pt}]

\newcommand{\rank}{\operatorname{rank}}

\newcommand{\Spec}{\operatorname{Spec}}

\newcommand{\Aut}{\operatorname{Aut}}

\newcommand{\pp}{\mathbb{P}}

\newcommand{\qq}{\mathbb{Q}}
\newcommand{\zz}{\mathbb{Z}}
\newcommand{\nn}{\mathbb{N}}
\newcommand{\rr}{\mathbb{R}}

\newcommand{\kk}{\mathbb{K}}

\def\O#1.{\mathcal {O}_{#1}}			
\def\pr #1.{\mathbb P^{#1}}				
\def\af #1.{\mathbb A^{#1}}			
\def\ses#1.#2.#3.{0\to #1\to #2\to #3 \to 0}	
\def\xrar#1.{\xrightarrow{#1}}			
\def\K#1.{K_{#1}}						
\def\bA#1.{\mathbf{A}_{#1}}			
\def\bM#1.{\mathbf{M}_{#1}}				
\def\bL#1.{\mathbf{L}_{#1}}				
\def\bB#1.{\mathbf{B}_{#1}}				
\def\bK#1.{\mathbf{K}_{#1}}			
\def\subs#1.{_{#1}}					
\def\sups#1.{^{#1}}

\DeclareMathOperator{\Supp}{Supp}

\usepackage{tikz}
\usetikzlibrary{matrix,arrows,decorations.pathmorphing}

\newtheorem{theorem}{Theorem}[section]
\newtheorem{introthm}{Theorem}

\newtheorem{lemma}[theorem]{Lemma}

\newtheorem{corollary}[theorem]{Corollary}

\newtheorem{notation}[theorem]{Notation}
\newtheorem{definition}[theorem]{Definition}

\newtheorem{question}[theorem]{Question}

\newtheorem{remark}[theorem]{Remark}

\theoremstyle{remark}

\numberwithin{equation}{section}

\usepackage[all]{xy}

\begin{document}
	
	\title[Complexity One Varieties are Cluster Type]{Complexity One Varieties are Cluster Type}
	
	\author[J. Enwright]{Joshua Enwright}
	\address{Department of Mathematics, UCLA, Mathematical Sciences Building, 520 Portola Plaza, Los Angeles, 90095, USA}
	\email{jlenwright1@math.ucla.edu}
	
	\author[J. Li]{Jennifer Li}
	\address{Department of Mathematics, Princeton University,
		Fine Hall, 304 Washington Road, Princeton, NJ}
	\email{jenniferli@princeton.edu}
	
	\author[J.I.~Y\'a\~nez]{Jos\'e Ignacio Y\'a\~nez}
	\address{Departamento de Matem\'atica, Universidad T\'ecnica Federico Santa Mar\'ia, Avenida Espa\~na 1680, Valpara\'iso, Chile}
	\email{jose.yaneze@usm.cl}
	
	
	
	\begin{abstract}
		The complexity of a Calabi--Yau pair $(X,B)$ is an invariant that relates the dimension of $X$, 
		the rank of the group of divisors, and the coefficients of $B$. 
		If the complexity is less than one, then $X$ is a toric variety. 
		We prove that if the complexity is less than two, then $X$ is a Fano type variety.
		Furthermore, if the complexity is less than 3/2, then $X$ admits a Calabi--Yau structure of complexity one and index at most two, and it admits a finite cover 
		$Y \to X$ of degree at most 2, where $Y$ is a cluster type variety.
		In particular, if the complexity is one and the index is one, $(X,B)$ is cluster type.
		Finally, we establish a connection with the theory of $T$-varieties. We prove that a variety of $T$-complexity one admits a similar finite cover from a cluster type variety.
	\end{abstract}
	
	\maketitle
	
	\setcounter{tocdepth}{1} 
	\tableofcontents
	
	\section{Introduction}
	Throughout this article, we work over an algebraically closed field $\mathbb{K}$ of characteristic zero.\par
	Let $X$ be a normal projective variety. The \textit{complexity} of a log pair $(X,B)$ is defined as
	$$c(X,B)=\dim X +\rank {\rm WDiv}_{\rm alg}(X)- |B|,$$
	where ${\rm WDiv}_{\rm alg}(X)$ is the group of Weil divisors on $X$ modulo algebraic equivalence and $|B|$ is the sum of the coefficients of $B$.
	While not particularly well-behaved for arbitrary log pairs, the complexity enjoys remarkable properties when restricted to certain sub-classes of pairs. In ~\cite[Theorem 1.2, Corollary 1.3]{BMSZ18}, Brown, McKernan, Svaldi and Zong show the following:
	
	\begin{introthm}[{\cite[Theorem 1.2, Corollary 1.3]{BMSZ18}}]\label{introthm:BMSZ18}
		Let $(X,B)$ be a log canonical pair with $-(K_X+B)$ nef. Then $c(X,B)\geq 0.$ If $c(X,B)<1,$ then there is a toric log Calabi--Yau pair $(X,\Delta)$ with $\lfloor B \rfloor \leq \Delta$ and all but possibly $\lfloor 2c(X,B)\rfloor$ components of $\Delta$ lie in the support of $B.$
	\end{introthm}
	
	Toric varieties occupy an interesting role in algebraic geometry. While they are, in some sense, extremely rare, they provide one of the most useful classes of varieties for hypothesis testing since they admit combinatorial descriptions that allow for explicit computation. Theorem ~\ref{introthm:BMSZ18} provides a useful criterion for detecting toric varieties in practice using only properties of log pairs. In particular, in ~\cite{EF24}, the first author and Figueroa show that if $c(X,B) = 0$, then $B$ is a weighted average of toric boundaries $\Delta_1,\ldots, \Delta_r$, corresponding to different choices of maximal tori in $\Aut(X)$. 
	
	Our first result generalizes Theorem \ref{introthm:BMSZ18} and~\cite[Lemma 2.4.3 and Theorem 3.2]{BMSZ18} to the case $c(X,B) < 2.$
	
	\begin{introthm}~\label{introthm:compl-less-than-2}
		Let $(X,B)$ be a log canonical pair with $-(K_X+B)$ nef and $c(X,B)<2.$ Then X is Fano type and $\Spec {\rm Cox}(X)$ has at worst a compound Du Val singularity at its vertex. If, moreover, $K_X+B\sim_{\rr}0,$ then the components of $B$ generate ${\rm Cl}(X)_\qq.$
	\end{introthm}
	
	We define the \textit{absolute complexity} of $X$ to be the quantity $$\widehat{c}(X)=\inf\{c(X,B)\mid (X,B) \text{ is log Calabi--Yau}\}.$$
	A consequence of Theorem ~\ref{introthm:BMSZ18} is that $\widehat{c}(X)\geq 0$ for all projective varieties $X,$ with $\widehat{c}(X)<1$ holding if and only if $X$ is toric. In fact, if $X$ is a projective toric variety and $B$ its reduced toric boundary, then $c(X,B)=0$ as both $\dim X +\rank {\rm WDiv}_{\rm alg}(X)$ and $|B|$ can be identified with the number of rays in the fan of $X.$ Thus, $\widehat{c}(X)<1$ if and only if $\widehat{c}(X)=0.$ We provide analogues to these results for higher values of the absolute complexity. In particular, we prove the following:
	\begin{introthm}\label{introthm:abs-compl-cA}
		Let $X$ be a normal projective variety. Then the following are equivalent:
		\begin{enumerate}
			\item $\widehat{c}(X)=1,$
			\item $\widehat{c}(X)\in \left[1,\frac{3}{2}\right),$
			\item $X$ is a non-toric variety of Fano type, and $\Spec {\rm Cox}(X)$ has a cA-type singularity at its vertex.
		\end{enumerate}
		If these equivalent conditions hold, then there is a reduced divisor $B$ on $X$ such that $(X,B)$ is a log Calabi--Yau pair of complexity one and index at most two.
	\end{introthm}
	
	We prove that $\frac{3}{2}$ is, indeed, the next value taken by the absolute complexity after the value $1.$
	
	\begin{introthm}\label{introthm:abs-compl-3/2}
		For every $n\geq 3,$ there is a Fano $n$-fold with ${\rm Cl}(X)=\zz$ such that $\widehat{c}(X)=\frac{3}{2}.$
	\end{introthm}
	
	In light of these results, it seems natural to ask what values the absolute complexity can take and through what ranges these values will be discrete. In this general direction we prove that, for varieties of Fano type, the infimum in the definition of absolute complexity is actually a minimum and is always a rational number.
	\begin{introthm}\label{introthm:abs-comp-is-min}
		Let $X$ be a Fano type variety. Then there exists a log Calabi--Yau pair $(X,B)$ with $B$ a $\qq$-divisor satisfying $c(X,B)=\widehat{c}(X).$ In particular, $\widehat{c}(X)\in \qq.$
	\end{introthm}
	
	Finally, we describe pairs $(X,B)$ of complexity one in terms of cluster type varieties 
	(see Definition~\ref{def:cluster-type}). Cluster type varieties are rational varieties that 
	generalize toric varieties and were introduced in~\cite{EFM24} by the first author, 
	Figueroa and Moraga. If $X$ is a $\qq$-factorial Fano variety and $(X,B)$ is a cluster type log Calabi--Yau pair of index one, then $X\setminus B$ is an affine variety covered up to codimension two by finitely many algebraic tori (\cite[Theorem 1.3(4)]{EFM24}, \cite[Theorem 23]{Cor23}). We show that a log Calabi--Yau pair $(X,B)$ with fine complexity $\overline{c}(X,B) \leq 1$ (see Definition \ref{def:fine-comp}) is a cluster type pair.
	
	\begin{introthm}\label{introthm:fine-compl-one-cluster-type}
		Let $(X,B)$ be a log Calabi--Yau pair of index one and $\overline{c}(X,B)\leq 1$ on a rationally connected variety $X.$ Then $(X,B)$ is cluster type.
	\end{introthm}
	
	Theorem \ref{introthm:fine-compl-one-cluster-type} fails if $X$ is not rationally connected. Indeed, let $E$ be a smooth curve of genus one. Then $(E,0)$ is a Calabi--Yau pair, and $\overline{c}(E,0) = 1,$ but this pair cannot be cluster type, as $E$ is not rational.
	
	The fine complexity $\overline{c}(X,B)$ of a pair $(X,B)$ is always bounded above by its complexity $c(X,B).$ In particular, since Fano type varieties are rationally connected ~\cite[Theorem 1]{Zha06}, it follows from Theorems ~\ref{introthm:compl-less-than-2} and ~\ref{introthm:fine-compl-one-cluster-type} that a log Calabi--Yau pair $(X,B)$ of index one and complexity one is cluster type. In this special case, we can provide an explicit description of the Cox ring of $X$ (see Lemma ~\ref{lem:cA-cox}). It is an example of a very simple \textit{Laurent phenomenon algebra} in the sense of ~\cite{LP16}. If, moreover, $X$ is $\mathbb{Q}$-factorial, then using Theorem ~\ref{introthm:compl-less-than-2} and arguing as in ~\cite{EFM24} one may show in this case that $X \setminus B$ is an affine variety covered up to codimension two by two algebraic tori. Following ~\cite[Example 2.9]{ducat24}, it follows that such an affine variety is also the spectrum of a Laurent phenomenon algebra. In \cite{MY24}, Moraga and the third author use cluster type varieties to describe log Calabi--Yau pairs of index one and complexity two.\par 
	If $\widehat{c}(X) = 1$, Theorem \ref{introthm:abs-compl-3/2} implies that there exists a reduced divisor $B$ on $X$ such that $(X,B)$ is log Calabi--Yau of index at most two, and $c(X,B) = 1$. We prove, in that case, that $X$ admits a finite morphism $Y\to X$ of degree at most 2, such that there exists a divisor $B_Y$ on $Y$ with $\overline{c}(Y,B_Y) \leq 1$. As another application of Theorem \ref{introthm:fine-compl-one-cluster-type}, we obtain the following result.
	
	\begin{introthm}\label{introthm:compl-one-quot-of-cluster-type}
		Let $X$ be a normal projective variety with $\widehat{c}(X)=1.$ Then there is a cluster type variety $Y$ and a surjective finite morphism $Y\rightarrow X$ of degree at most $2.$
	\end{introthm}
	In light of Theorem ~\ref{introthm:abs-compl-cA}, Theorem ~\ref{introthm:compl-one-quot-of-cluster-type} can be viewed as a strengthening of ~\cite[Theorem 1.8]{BMSZ18}.
	
	A \emph{$T$-variety} is a normal projective variety $X$ with an effective action of an algebraic torus $T\cong \mathbb{G}_m^k$. 
	These varieties were introduced by Altmann and Hausen in \cite{AH06}, where they described $T$-varieties  via ``polyhedral divisor'', generalizing the theory of toric varieties and varieties with a $\mathbb{G}_m$-action.
	Several results of toric varieties have been generalized to the setting of $T$-varieties. For example, results about the fundamental group \cite{LLM19}, singularities \cite{LS13}, and Cox rings \cite{HS10}.
	The \emph{$T$-complexity} of $X$ is the difference $\dim X - k$, which is the smallest codimension in $X$ of an orbit of the $T$-action. The $T$-varieties of $T$-complexity zero are precisely the toric varieties. We prove that a Fano type variety $X$ of dimension $n$ and $T$-complexity one admits a finite morphism $Y \to X$, with $(Y,B_Y)$ a log Calabi--Yau sub-pair and $Y\setminus \Supp B_Y$ contains an open algebraic torus of dimension $n$.
	
	\begin{introthm}\label{introthm:t-compl-one}
		Let $X$ be an $n$-dimensional Fano type variety of $\mathbb{T}$-complexity one. Then there exists a surjective finite morphism $Y\rightarrow X$ of degree at most $60$ from a normal projective variety $Y$ that admits a log Calabi--Yau sub-pair $(Y,B_Y)$ and an open embedding $\mathbb{G}_m^n\hookrightarrow Y\setminus \Supp B_Y.$
	\end{introthm}
	
	The number $60$ in the statement of Theorem ~\ref{introthm:t-compl-one} shows up as the order of the icosohedral group in its guise as the orbifold fundamental group of the smooth orbifold $(\pp^1, \frac{1}{2}p+\frac{2}{3}q+\frac{4}{5}r)$ (see ~\cite[Definition 1.2]{Clau08}). In the event that the sub-boundary $B_Y$ is effective, it will follow from the proof that $(Y,B_Y)$ is a cluster type pair. In general, one might refer to $(Y,B_Y)$ as a \textit{cluster type sub-pair}.
	
	\subsection*{Acknowledgements}
	The authors would like to thank J\'{a}nos Koll\'{a}r, Joaquín Moraga and Burt Totaro for valuable comments and discussions. The first author is partially supported by NSF grant DMS-2136090. Part of this project was done during the participation of the first and second authors in the \href{https://aimath.org/pastworkshops/higherdimlogcy.html}{``Higher-dimensional log Calabi-Yau pairs''} workshop at the American Institute of Mathematics.
	
	\section{Preliminaries}
	We work over an algebraically closed field $\kk$ of characteristic zero. In this section we will recall definitions and results related to the Minimal Model Program and its singularities, the complexity, cluster type pairs, conic fibrations, and Cox rings and Mori Dream Spaces.
	
	\subsection{Log pairs} For definitions of the Minimal Model Program and its singularities, we refer to \cite{KM98} and \cite{Kol13}. We recall some of the definitions for the reader's convenience.
	
	\begin{definition}{\em
			A \textit{log sub-pair} $(X,B)$ consists of a normal quasi-projective variety $X$ and a $\rr$-divisor $B$ with the property that $K_X+B$ is $\rr$-Cartier. We say that a log sub-pair $(X,B)$ is a \textit{log pair} if $B$ is effective.}
	\end{definition}
	
	\begin{definition}{\em
			Let $(X,B)$ be a log sub-pair, and let $f\colon Y\rightarrow X$ be a projective birational morphism from a normal variety $Y$. We will refer to the unique log sub-pair $(Y,B_Y)$ satisfying
			\begin{itemize}
				\item $K_Y+B_Y\sim_\rr f^*(K_X+B),$ and
				\item $f_*B_Y=B$
			\end{itemize}
			as the \textit{log pullback of $(X,B)$ via $f.$}}
	\end{definition}
	
	\begin{definition}{\em
			Let $(X,B)$ be a log sub-pair and let $f\colon Y \rightarrow X$ be a projective birational morphism from a normal variety $Y.$ Given a prime divisor $E\subset Y,$ its \textit{log discrepancy} with respect to $(X,B)$ is the quantity
			$$a_E(X,B)=1-{\rm coeff}_E(B_Y),$$
			where $(Y,B_Y)$ is the log pullback of $(X,B)$ via $f.$\par We say that a log pair $(X,B)$ is \textit{log canonical} (respectively \textit{Kawamata log terminal}) if $a_E(X,B)\geq 0$ (respectively $a_E(X,B)> 0$) for all prime divisors $E$ over $X.$}
	\end{definition}
	
	\begin{definition}{\em
			Let $(X,B)$ be a log canonical pair. A \textit{log canonical place} of $(X,B)$ is a divisor $E$ over $X$ for which $a_E(X,B)=0$. A \textit{log canonical center} of $(X,B)$ is a subvariety $Z\subset X$ which is the image on $X$ of a log canonical place of $(X,B)$.}
	\end{definition}
	
	\begin{definition}\label{def:dlt}
		{\em
			Let $(X,B)$ be a log pair. We say that $(X,B)$ is \textit{divisorially log terminal} if it is log canonical and there exists an open subset $U \subset X$, such that:
			\begin{enumerate}
				\item the pair $(U,B_U)$ is log smooth, and
				\item every log canonical center of $(X,B)$ intersects $U.$
			\end{enumerate}
		}
	\end{definition}
	
	\begin{notation}{\em
			We will often abbreviate log canonical, Kawamata log terminal and divisorially log terminal as lc, klt and dlt respectively.}
	\end{notation}
	
	\begin{definition}\label{def:dlt-mod}
		{\em Let $(X,B)$ be a log canonical pair.  A \textit{dlt modification} of $(X,B)$ is a projective birational morphism $f \colon Y \rightarrow X$, satisfying the following conditions:
			
			\begin{enumerate}
				\item $Y$ is $\qq$-factorial,
				\item every $f$-exceptional divisor is a log canonical place of $(X,B)$, and
				\item the log pullback $(Y,B_Y)$ of $(X,B)$ via $f$ is dlt.
			\end{enumerate}
		}
	\end{definition}
	
	We recall the existence of dlt modifications.
	
	\begin{lemma}\label{lem:exist-dlt-mod}
		Let $(X,B)$ be a log canonical pair. Then there exists a dlt modification $f\colon Y\rightarrow X$ of $(X,B)$.
	\end{lemma}
	
	\begin{proof}
		This is ~\cite[Corollary 3.6]{Bir12}.
	\end{proof}
	
	\begin{definition}\label{def:logCY}
		{\em Let $(X,B)$ be a log pair. We say that $(X,B)$ is a \textit{log Calabi--Yau pair} if $X$ is projective, $(X,B)$ is log canonical and $K_X+B\sim_{\rr} 0$.
		}
	\end{definition}
	
	\begin{definition}
		{\em
			Let $(X,B)$ be a log Calabi--Yau pair with $B$ a $\qq$-divisor. We define the \emph{index of (X,B)} as the smallest positive integer $m$ for which $m(K_X + B) \sim 0.$
		}
	\end{definition}
	
	\begin{definition}\label{def:crepant}
		{\em
			Let $(X,B)$ and $(Y,B_Y)$ be two log sub-pairs. We will say that a birational map $f\colon X \dashrightarrow Y$ is \textit{crepant} with respect to these sub-pairs if it admits a resolution 
			\[
			\xymatrix{ 
				& Z \ar[ld]_-{p} \ar[rd]^-{q} & \\ 
				X \ar@{-->}[rr]^-{f}& & Y
			}
			\]
			with proper birational morphisms $p$ and $q$ such that the log pullback of $(X,B)$ via $p$ is equal to the log pullback of $(Y,B_Y)$ via $q.$
		}
	\end{definition}
	
	\begin{remark}\label{rmk:crepant-discrepancy}
		{\em
			Let $f\colon (X,B)\dashrightarrow (Y,B_Y)$ be a crepant birational map between two log pairs. It follows from ~\cite[Lemma 2.30]{KM98} that $a_E(X,B)=a_E(Y,B_Y)$ for every divisor $E$ over $X$ and $Y.$
		}
	\end{remark}
	
	\begin{definition}\label{def:crep-fin}
		{\em
			Let $f\colon X\rightarrow Y$ be a surjective finite morphism between normal quasi-projective varieties, and let $(X,B)$ and $(Y,B_Y)$ be log sub-pairs. We will say that $f$ is \textit{crepant} with respect to $(X,B)$ and $(Y,B_Y)$ if for all prime divisors $D\subset X$ we have $${\rm coeff}_D(B)=1-r_D\left(1-{\rm coeff}_{D_Y}(B_Y)\right),$$
			where $D_Y=f(D)$ and $r_D$ is the ramification index of $f$ at the generic point of $D.$
		}
	\end{definition}
	\begin{remark}\label{rem:crep-pullback-fin}
		{\em
			If $f\colon X \rightarrow Y$ is a surjective finite morphism between normal quasi-projective varieties that is crepant with respect to log sub-pairs $(X,B)$ and $(Y,B_Y),$ then it follows that $$f^*(K_Y+B_Y)\sim_{\rr}K_X+B.$$
		}
	\end{remark}
	\begin{definition}\label{def:crep-gen-fin}
		{\em
			Let $f\colon X\rightarrow Y$ be a projective, surjective and generically finite morphism between normal quasi-projective varieties, and let $(X,B)$ and $(Y,B_Y)$ be log sub-pairs. Let $$X\xrightarrow{g}W\xrightarrow{h}Y$$
			be the Stein factorization of $f,$ and write $B_W=f_*B.$ We will say that $f$ is \textit{crepant} with respect to the pairs $(X,B)$ and $(Y,B_Y)$ if the birational morphism $g$ is crepant with respect to $(X,B)$ and $(W,B_W)$ and if the finite morphism $h$ is crepant with respect to $(W,B_W)$ and $(Y,B_Y).$
		}
	\end{definition}

	\subsection{Complexity}
	
	In this subsection we recall the notion of complexity and some basic facts about its behavior under contracting birational maps.
	
	\begin{definition}\label{defn:complexity}{\em
			Let $X$ be a normal projective variety and let $(X,B)$ be a log sub-pair. The \textit{complexity} of $(X,B)$ is $$c(X,B)=\dim X + \rank {\rm WDiv}_{\rm alg}(X)-|B|,$$
			where ${\rm WDiv}_{\rm alg}(X)$ is the group of Weil divisors on $X$ modulo algebraic equivalence and $|B|$ is the sum of the coefficients of $B$. 
		}
	\end{definition}
	
	The following Lemma allows us to compare the behavior of the complexity under a birational contraction (see Definition \ref{def:bir-contraction-map}).
	
	\begin{lemma}[{\cite[Lemma 2.41]{EF24}}]\label{lem:complexity-and-exceptional-divisors}
		Let $f\colon X \dashrightarrow Y$ be a contracting birational map between normal projective varieties. Let $(X,B)$ be a log sub-pair. Denote by $E_1,\hdots, E_r$ the prime $f$-exceptional divisors and write $B_Y=f_*B.$ Then $$c(X,B)=c(Y,B_Y)+\sum_{i=1}^ra_{E_i}(X,B).$$ 
	\end{lemma}
	
	As a corollary, we see that the complexity of a log canonical pair is unaffected by extracting log canonical places.
	
	\begin{corollary}\label{cor:compl-dlt-mod}
		Let $(X,B)$ be a log canonical pair and let $f\colon (Y,B_Y)\rightarrow (X,B)$ be a birational morphism extracting only log canonical places of $(X,B)$. Then $$c(X,B)=c(Y,B_Y).$$
	\end{corollary}
	\begin{proof}
		Denote by $E_1,\hdots, E_r$ the prime $f$-exceptional divisors. By assumption, $a_{E_i}(Y,B_Y)=0$ for all $1\leq i \leq r.$ The desired result follows from Lemma ~\ref{lem:complexity-and-exceptional-divisors}.
	\end{proof}
	
	The following variant of the complexity, which we will refer to as the \textit{fine complexity}, is what was referred to as the complexity in ~\cite[Definition 1.1]{BMSZ18}.
	
	\begin{definition}\label{def:fine-comp}
		{\em
			Let $X$ be a normal projective variety, and let $(X,B)$ be a log pair. A \textit{decomposition} $\Sigma$ of $B$ is an expression of the form $\sum_{i=1}^ka_iB_i\leq B,$ where $B_i$ is an integral Weil divisor and $a_i\geq 0$ for each $1\leq i \leq k.$ Given a decomposition $\Sigma$ of $B,$ its \textit{norm} is $|\Sigma|=\sum_{i=1}^ka_i$, its \textit{rank} $\rho(\Sigma)$ is the rank of the subgroup spanned by $B_1,\hdots, B_k$ in ${\rm WDiv}_{\rm alg}(X),$ and its \textit{complexity} is the quantity $c(\Sigma)=\dim X + \rho(\Sigma)-|\Sigma|.$ The \textit{fine complexity} of the log pair $(X,B)$ is the quantity $$\overline{c}(X,B)=\inf\{c(\Sigma)\mid \Sigma\text{ is a decomposition of }B\}.$$
		}
	\end{definition}
	
	We recall the following definition from the introduction.
	\begin{definition}\label{def:abs-comp}
		{\em
			Let $X$ be a normal projective variety $X,$ its \textit{absolute complexity} is the quantity
			$$\widehat{c}(X)=\inf\{c(X,B)\mid (X,B)\text{ is log Calabi--Yau}\}.$$
		}
	\end{definition}
	
	As a consequence of Lemma ~\ref{lem:complexity-and-exceptional-divisors}, we obtain the following basic fact regarding the behavior of absolute complexity under birational contractions.
	\begin{lemma}\label{lem:abs-comp-bir-contr}
		Let $f\colon X\dashrightarrow Y$ be a contracting birational map between normal projective varieties. Then $\widehat{c}(Y)\leq \widehat{c}(X).$
	\end{lemma}
	\begin{proof}
		This is automatic if $\widehat{c}(X)=\infty,$ so it suffices to show the result when $\widehat{c}(X)<\infty.$ Given any log Calabi--Yau pair $(X,B)$ supported on $X,$ we obtain $$\widehat{c}(Y)\leq c(Y,f_*B)\leq c(X,B)$$
		by Lemma ~\ref{lem:complexity-and-exceptional-divisors}. The desired inequality follows.
	\end{proof}
	
	\subsection{Cluster type pairs}
	
	In this subsection we define cluster type pairs and cluster type varieties.
	
	\begin{definition}[c.f. {\cite[Definition 2.26]{EFM24}}]\label{def:cluster-type}
		{\em
			We say that a log Calabi--Yau pair $(X,B)$ is a \textit{cluster type pair} if there exists a toric log Calabi--Yau pair $(T,B_T)$ and a crepant birational map $\phi\colon (T,B_T)\dashrightarrow (X,B)$
			such that all $\phi$-exceptional divisors have coefficient one in $(T,B_T).$ We say that a variety $X$ is \textit{cluster type} if it admits a log Calabi--Yau pair $(X,B)$ of cluster type.
		}
	\end{definition}
	
	\begin{remark}
		\label{rem:Corti-clusterVariety}
		We may think of a cluster type variety $X$ as a projective variety $X$ such that $X \setminus B$ is a cluster variety in the sense of Corti (see \cite{Cor23}, Definition 2)
	\end{remark}
	The following observation follows immediately from the definition.
	\begin{lemma}\label{lem:check-cluster-type-on-bir-model}
		Let $(X',B')\dashrightarrow (X,B)$ be a crepant birational map between log Calabi--Yau pairs extracting only log canonical places of $(X,B).$ If $(X',B')$ is cluster type, then so is $(X,B).$ 
	\end{lemma}
	
	\subsection{Contractions, Fibrations and Canonical Bundle Formula}\label{subsec:contractions} In this subsection we recall the definition of contractions and fibrations.
	
	\begin{definition}{\em
			Let $f\colon X \rightarrow Y$ be a proper morphism of varieties. We denote by $$N^1(X/Y)=\left({\rm Pic}(X)\otimes_\zz\rr\right)/\equiv_Y$$
			and $$N_1(X/Y)=\left(Z_1(X/Y)\otimes_\zz\rr\right)/\equiv_Y$$
			the real vector spaces of Cartier divisors and relative $1$-cycles modulo numerical equivalence over $Y,$ respectively (see ~\cite[Section IV.4]{Kle66}). These vector spaces are dual under the intersection pairing, and are finite-dimensional by ~\cite[Proposition IV.4.3]{Kle66}. Their common dimension is denoted by $\rho(X/Y)$ and is referred to as the \textit{relative Picard rank} of the morphism $f.$
		}
	\end{definition}
	
	\begin{remark}{\em
			When $Y=\Spec \kk$ is a point, we will simply write $N^1(X), N_1(X)$ and $\rho(X)$ and will omit the word ``relative''.
		}
	\end{remark}
	
	\begin{definition}\label{def:contraction}{\em
			We say that a morphism $f\colon X \rightarrow Y$ between normal quasi-projective varieties is a \textit{contraction} if it is projective and satisfies $f_*\mathcal{O}_X=\mathcal{O}_Y$. We say that the contraction $f$ is a \textit{fibration} if $\dim Y < \dim X.$ We say that the contraction $f$ is \textit{extremal} if $\rho(X/Y)=1.$
		}
	\end{definition}
	
	\begin{remark}{\em
			A contraction $f\colon X \rightarrow Y$ is birational if and only if $\dim Y = \dim X,$ and a projective birational morphism between normal varieties is automatically a contraction.
		}
	\end{remark}
	
	\begin{definition}\label{def:bir-contraction-map}{\em
			We say that a birational map $f\colon X \dashrightarrow Y$ between normal quasi-projective varieties is a \textit{contracting birational map} if it is surjective in codimension one, that is, if $f$ is a morphism on some open subset $U\subset X$ such that $Y\setminus f(U)$ has codimension at least two.
		}
	\end{definition}
	
	\begin{definition}\label{def:horiz-vert-divs}
		{\em
			Let $f\colon X\rightarrow Y$ be a fibration and let $D\subset X$ be a prime divisor. We say that $D$ is \textit{horizontal} over $Y$ if $f(D)=Y.$ Otherwise, we say that $D$ is \textit{vertical} over $Y.$
		}
	\end{definition}

	Let $f\colon X \to Y$ be a fibration and let $(X,B)$ be log canonical pair with $K_X + B \sim_{f,\qq} 0$. By the Canonical Bundle Formula, there exists an effective $\qq$-divisor $B_Y$ and a nef $\qq$-b-divisor $\mathbf{M}$ on $Y$ such that $K_X + B \sim_\qq f^*(K_Y + B_Y + \mathbf{M}_Y)$. We call $\mathbf{M}$ {\em the moduli b-divisor} and $(Y,B_Y+\mathbf{M})$ the {\em generalized pair induced via the canonical bundle formula}. For more details on the canonical bundle formula and generalized pairs, we refer the reader to \cite{Fil18}.

	\begin{lemma}\label{lem:div-image-of-lcc}
		Let $(X,B)$ be a log Calabi--Yau pair with $B$ a $\qq$-divisor and let $f\colon X\rightarrow Y$ be a fibration. Let $(Y,B_Y+\mathbf{M})$ be the generalized pair induced on $Y$ via the canonical bundle formula, and let $D\subset Y$ be an irreducible component of $\lfloor B_Y \rfloor.$ Then there exists a log canonical center $Z\subset X$ of $(X,B)$ such that $f(Z)=D.$
	\end{lemma}
	\begin{proof}
		We will assume that there is no such log canonical center of $(X,B)$ and will derive a contradiction. Denote by $W\subset X$ be the union of all log canonical centers of $(X,B)$ that do not dominate $Y.$ Then $U=Y_{{\rm reg}}\setminus f(W)$ is an open subset of $Y$ that intersects $D$ nontrivially. Write $V=f^{-1}(U)$ and $E=f|_V^{-1}(D|_U).$ We have that $\Supp(E)$ contains no log canonical centers of $(V,B|_V),$ hence that $(V,B_V+\epsilon E)$ is log canonical for $\epsilon>0$ sufficiently small. It follows from this and the definition of the discriminant divisor $B_Y$ that ${\rm coeff}_D(B_Y)<1,$ contradicting the fact that $D$ is a component of $\lfloor B_Y \rfloor.$
	\end{proof}
	
	\begin{lemma}\label{lem:no-moduli-gen-finite-lcc}
		Let $f\colon X \rightarrow Y$ be a fibration between normal projective varieties. Let $(X,B)$ be a log canonical pair that satisfies $K_X+B\sim_{\mathbb{Q},f}0,$ and write $(Y,B_Y+\mathbf{M})$ for the generalized sub-pair induced on $Y$ via the canonical bundle formula. Assume that there exists a component $S\subset \lfloor B \rfloor$ such that $f|_S\colon S \rightarrow Y$ is surjective and generically finite, and write $B_S={\rm Diff}_S(B)$. Then the following hold:
		\begin{enumerate}
			\item the moduli b-divisor $\mathbf{M}$ is $\qq$-linearly trivial as a b-divisor, and 
			\item the morphism $f|_S$ is crepant with respect to the pairs $(S,B_S)$ and $(Y,B_Y)$.
		\end{enumerate}
	\end{lemma}
	\begin{proof}
		This follows from ~\cite[Theorem 5.4]{Fil18}.
	\end{proof}

	\subsection{Conic Fibrations}
	
	In this subsection we prove some results about conic fibrations.
	
	\begin{definition}\label{defn:conic-fib}
		{\em
			Let $f\colon X \rightarrow Y$ be a fibration between normal projective varieties. We will call $f$ a \textit{conic fibration} if a general fiber of $f$ is isomorphic to $\pp^1.$
		}
	\end{definition}
	
	\begin{lemma}\label{lem:conic-fib-horiz-divs}
		Let $f\colon X \rightarrow Y$ be a conic fibration between normal, $\qq$-factorial projective varieties. Suppose that $(X,B)$ is a log canonical pair such that $K_X+B\sim_{\qq,Y}0.$ Then at least one component of $B$ is horizontal over $Y.$
	\end{lemma}
	\begin{proof}
		We will assume that $B$ has no horizontal components and will obtain a contradiction. It follows from this assumption and Definition ~\ref{defn:conic-fib} that there is a nonempty open $U\subset Y_{\rm reg}$ such that $V=f^{-1}(U)$ is disjoint from the support of $B$ and over which every fiber is isomorphic to $\pp^1.$ Fix $p\in U,$ write $F_p=f^{-1}(p).$ Then $K_{F_p}=K_V|_{F_p}.$ But $K_V\sim_{\qq,U}-B|V=0,$ so $K_{F_p}\sim_\qq 0.$ But this is nonsense, as $F_p\cong \pp^1.$
	\end{proof}
	
	\begin{lemma}\label{lem:conic-fib-to-mfs}
		Let $X$ be a projective variety of klt type and let $f\colon X \rightarrow Y$ be a conic fibration. Then there exists a commutative diagram 
		\[
		\xymatrix{  
			X \ar[rd]_-{f} \ar@{-->}[rr]^-{\phi} & & X' \ar[ld]^-{f'}\\
			& Y &
		}
		\]
		satisfying the following conditions:
		\begin{enumerate}
			\item $\phi$ is a contracting birational map,
			\item $X'$ is klt and $\qq$-factorial, and
			\item $f'$ is a Mori fiber space.
		\end{enumerate}
	\end{lemma}
	
	\begin{proof}
		Using ~\cite[Corollary 1.37]{Kol13}, we may assume that $X$ is $\qq$-factorial. Since $K_X$ is not $f$-pseudo-effective, it follows from ~\cite[Corollary 1.3.3]{BCHM10} that we may assume that $f$ factors as 
		$$X\xrightarrow{g}Z\xrightarrow{h}Y,$$
		where $g$ is a Mori fiber space and $h$ is a birational morphism. The desired result now follows from ~\cite[Lemma 2.11]{MM24}.
	\end{proof}
	
	\begin{lemma}\label{lem:conic-fib-plt-pairs}
		Let $f\colon X \rightarrow Y$ be a conic fibration between projective varieties, with $Y$ smooth and with $X$ klt and $\qq$-factorial. Let $B$ be a reduced, effective divisor on $X$ such that $(X,B)$ is log canonical and $K_X+B\sim_{\qq,Y}0.$ Suppose that the pair induced on $Y$ via the canonical bundle formula is a log smooth pair $(Y,B_Y)$ with $B_Y$ reduced. Then there exists a commutative diagram 
		\[
		\xymatrix{  
			(X,B) \ar[rd]_-{f} & & (X',B') \ar@{-->}[ll]_-{\phi} \ar[ld]^-{f'}\\
			& Y &
		}
		\]
		satisfying the following conditions:
		\begin{enumerate}
			\item $\phi$ is a crepant birational map extracting only log canonical places of $(X,B),$
			\item $f'$ is a flat morphism,
			\item $(X',B'_{\rm horiz})$ is plt, and
			\item no component of $B'_{\rm horiz}$ contains a fiber of $f'.$
		\end{enumerate}
	\end{lemma}
	
	\begin{remark} 
		{\em
			In the notation above, we write $B^{\prime}_{\rm horiz}$ to mean all of the horizontal components of $B^{\prime}$. Similarly, we use $B^{\prime}_{\rm vert}$ to mean all of the vertical components of $B^{\prime}$.
		}
	\end{remark}
	
	\begin{proof}
		Using Lemma ~\ref{lem:conic-fib-to-mfs}, we may assume that $f$ is a Mori fiber space. By Lemma ~\cite[Lemma 2.10]{MM24}, there exists a commutative diagram 
		\[
		\xymatrix{  
			(X,B) \ar[rd]_-{f} & & (X',B') \ar@{-->}[ll]_-{\phi} \ar[ld]^-{f'}\\
			& Y &
		}
		\]
		satisfying the following conditions:
		\begin{enumerate}
			\item[(i)] $\phi$ is a crepant birational map extracting only log canonical places of $(X,B),$
			\item[(ii)] $X'$ is $\qq$-factorial,
			\item[(iii)] $f'$ is a Mori fiber space,
			\item[(iv)] $f'^{-1}(B_Y)\subset B'.$
		\end{enumerate}
		Condition (1) is exactly (i) above. We will show that conditions (2)-(4) are satisfied by this $(X',B')$ and $f'.$\par
		Consider an arbitrary point $p\in Y.$ Since $(Y,B_Y)$ is log smooth and $B_Y=\lfloor B_Y \rfloor,$ it follows that there is a reduced divisor $D$ sharing no components with $B_Y$ such that $(Y,B_Y+D)$ is a log smooth pair that has $p$ as a log canonical center. It follows from ~\cite[Remark 3.1.6]{Amb99} and ~\cite[Proposition 4.16]{Fil18} that the pair $(X',B'+f'^*D)$ is a log canonical pair with a reduced boundary. Write $n=\dim X'.$ Thus, there are $n-1$ components of $B_Y+D$ that contain $p.$ It follows from this and condition (iv) that there are at least $n-1$ components of $B'_{\rm vert}+f'^*D$ that contain the fiber $f'^{-1}(p).$ Denote by $c\geq 0$ the number of components of $B'_{\rm horiz}$ that contain $f'^{-1}(p)$ and by $m\geq 1$ the dimension of $f'^{-1}(p).$ Consider general very ample effective Cartier divisors $A_1,\hdots, A_m$ on $X'.$ Then $(X',B'+f'^*D+\sum_{i=1}^mA_i)$ is a log canonical pair with reduced boundary, and there is a point $q\in f'^{-1}(p)$ that is contained in at least $n-1+c+m$ components of $B'+f'^*D+\sum_{i=1}^mA_i$. It follows from ~\cite[Lemma 2.4.3]{BMSZ18} that $c=0$ and $m=1$. In particular, this tells us that no ($c = 0$) components of $(B')_{\rm horiz}$ contain $(f')^{-1}(p)$; and the dimension of $(f')^{-1}(p)$ is $m=1$. Since $p \in Y$ was chosen arbitrarily, we see that $f'$ is equidimensional and condition (4) holds. Condition (2) follows from the equidimensionality of $f'$ by miracle flatness. Indeed, $Y$ is smooth and $X$ is Cohen-Macaulay.\par
		To show condition (3), we will assume that there exists an exceptional log canonical place $E$ of $(X',B'_{\rm horiz})$ and will derive a contradiction. It follows from condition (iii) that $B'_{\rm horiz}\sim_{\qq,Y}B',$ hence that $K_{X'}+B'_{\rm horiz}\sim_{\qq,Y}0.$ It follows from Lemmas ~\ref{lem:no-moduli-gen-finite-lcc} and ~\ref{lem:conic-fib-horiz-divs} that that the generalized pair induced on $Y$ from $(X',B'_{\rm horiz})$ via the canonical bundle formula has trivial moduli part. Write $(Y, \widetilde{B}_Y)$ for the pair induced by $(X',B'_{\rm horiz})$ via the canonical bundle formula. Since $B'_{\rm horiz}\leq B',$ it follows that we have $\widetilde{B}_Y\leq B_Y.$ It follows from this inequality and condition (iv) that $\lfloor \widetilde{B}_Y \rfloor=0.$ Since $(Y,B_Y)$ is log smooth, we see that the pair $(Y,\widetilde{B}_Y)$ is klt. \par 
		Using ~\cite[Corollary 1.38]{Kol13}, we choose a dlt modification $g\colon (X'',B'')\rightarrow (X'B'_{\rm horiz})$ which extracts the divisor $E$. We then apply ~\cite[Theorem 2.1, Proposition 4.4]{AK00} to obtain a commutative diagram 
		\[
		\xymatrix{ 
			(X'',B'')\ar[d]_-{f'\circ g} & (X''',B''')\ar[l]_-{\beta}\ar[d]^-{f'''} \\ 
			Y &
			(Y',B_Y')\ar[l]_-{\gamma}
		}
		\]
		in which the horizontal morphisms are crepant birational and such that $f'''$ is equidimensional. Denoting by $E'''$ the center of $E$ on $X''',$ if follows that $f'''(E''')$ is a divisor on $Y'.$ Since ${\rm coeff}_{E'''}(B''')=1,$ it follows from the definition of the discriminant divisor $B'_Y$ that ${\rm coeff}_{f'''(E''')}(B'_Y)=1.$ This implies that $f'''(E''')$ is a log canonical place of $Y,$ which is nonsense as $Y$ is smooth.

	\end{proof}

	\begin{lemma}\label{lem:compl-one-horiz-divs}
		Let $(X,B)$ be a log Calabi--Yau pair of index one such that $c(X,B)=1,$ and let $f\colon X \rightarrow Y$ be a conic fibration. Then $B$ has two components that are horizontal over $Y.$
	\end{lemma}
	\begin{proof}
		Using ~\cite[Corollary 1.36]{Kol13} if necessary, we may assume that $X$ is $\qq$-factorial. Replacing $Y$ with the base of a Mori fiber space obtained as the output of a $K_X$-MMP of $Y,$ we may assume that $Y$ is $\qq$-factorial. It follows from Lemma ~\ref{lem:conic-fib-horiz-divs} that at least one component of $B$ is horizontal over $Y.$ We will assume that exactly one component of $B$ is horizontal over $Y$ and will derive a contradiction. Denote this horizontal component by $S\subset B.$ Since $S$ is a log canonical center of $(X,B)$ that dominates $Y$ and satisfies $\dim S = \dim Y,$ it follows from Lemma ~\ref{lem:no-moduli-gen-finite-lcc} and ~\cite[Proposition 4.16]{Fil18} that the pair $(Y,B_Y)$ obtained from the canonical bundle formula is a log Calabi--Yau pair. There are $\dim X + \rho(X)-2=\dim Y + \rho(Y)$ components of $B$ that are vertical over $Y.$ Since $\rho(X/Y)=1,$ this implies that there are at least $\dim Y + \rho(Y)$ components of $\lfloor B_Y \rfloor.$ Since $c(Y,B_Y)\geq 0,$ it follows from this that $B_Y=\lfloor B_Y\rfloor$ and $c(Y,B_Y)=0.$ Thus, $(Y,B_Y)$ is a toric log Calabi--Yau pair.\par
		Since $(Y,B_Y)$ is a toric log Calabi--Yau pair, we may choose a crepant birational map $\phi \colon (\pp^{n-1},\Sigma^{n-1}) \dashrightarrow (Y,B_Y)$ extracting only log canonical places of $(Y,B_Y).$ Here, $n=\dim X.$ By ~\cite[Lemma 2.12]{MM24}, there is a commutative diagram 
		\[
		\xymatrix{ 
			(X,B)\ar[d]_-{f} & (X',B_{X'})\ar@{-->}[l]_-{\psi}\ar[d]^-{g} \\ 
			(Y,B_Y) &
			(\pp^{n-1},\Sigma^{n-1})\ar@{-->}[l]_-{\phi}
		}
		\]
		satisfying the following conditions:
		\begin{enumerate}
			\item $X'$ is $\qq$-factorial,
			\item $\psi$ is a crepant birational map extracting only log canonical places of $(X,B),$
			\item $\psi$ is an isomorphism over the generic point of $Z,$
			\item $g^{-1}(\Sigma^{n-1})\subset B_{X'},$ and
			\item $g$ is a Mori fiber space.
		\end{enumerate}
		It follows that $(X',B_{X'})$ is a log Calabi--Yau pair of index one of complexity one, and that there is exactly one component $S'\subset B_{X'}$ which is horizontal over $\pp^{n-1}.$ Thus, we may assume that $Y=\pp^{n-1}$ and that $B_Y=\Sigma^{n-1}$.\par
		By Lemma ~\ref{lem:conic-fib-plt-pairs}, we may assume that $(X,S)$ is a plt pair, hence that both $X$ and $S$ are klt, and that the induced morphism $f_S\colon S \rightarrow Y$ is finite. We note that this finite morphism is of degree two. Denoting by $(S,B_S)$ the log Calabi--Yau pair obtained from $(X,B)$ via adjunction, it follows from Lemma ~\ref{lem:no-moduli-gen-finite-lcc} that $f_S\colon (S,B_S)\rightarrow (Y,B_Y)$ is a crepant finite morphism. Since $Y$ is smooth, it follows that $f_S$ is \'{e}tale over $Y\setminus B_Y.$ Since $Y$ is simply connected, it follows that $f_S$ must ramify over some component of $B_Y.$ Without any loss of generality, we may assume that it ramifies over $H_0.$\par
		For $0\leq i \leq n-1,$ write $Z_i=f^{-1}(H_0)_{\rm red}.$ By the previous paragraph, $W=(Z_0\cap S)_{\rm red}$ is irreducible and maps onto $H_0$ with degree one. Let $L\subset \pp^{n-1}$ be a general line and let $T=f^{-1}(L)$. It follows that $T$ is a klt surface, and we obtain, via adjunction from $(X,B),$ a log canonical pair $(T,B_T)$ with $K_T+B_T$ Cartier. The components of $B_T$ are $S_T,$ $Z_{T,0},\hdots, Z_{T,n-1},$ where $S_T=S\cap T$ and $Z_{T,i}=Z_i\cap T$ for $0\leq i \leq n-1.$ The intersection $Z_{T,0}\cap S_T$ consists of the single point $W\cap T,$ and $Z_{T,0}$ does not intersect $Z_{T,i}$ for any $i\neq 0.$ Since $K_T+B_T$ is Cartier, it follows from ~\cite[Theorem 2.1.2]{Pro01c} that both $T$ and $Z_{T,0}$ are smooth along $Z_{T,0}$ away from $W\cap T.$ It follows that $Z_0$ can contain no log canonical centers of $(X,B)$ that are horizontal over $H_0$ aside from $W$ and itself. \par
		Let $\widetilde{Z}_0$ be the normalization of $Z_0$, and denote by $(\widetilde{Z}_0, B_{\widetilde{Z}_0})$ the log Calabi--Yau pair induced via adjunction from $(X,B).$ On the one hand, the morphism $\widetilde{Z}_0\rightarrow H_0$ is a conic fibration. Thus, we must have $B_{\widetilde{Z}_0}\cdot F=2$ for a general fiber $F.$ On the other hand, the strict transform $\widetilde{W}$ of $W$ is the only horizontal component of $B_{\widetilde{Z}_0}.$ Since $\widetilde{W}$ maps onto $H_0$ with degree one, it follows that we must have $B_{\widetilde{Z}_0}\cdot F=1$ for a general fiber. We have obtained the desired contradiction.
	\end{proof}
	
	\begin{lemma}\label{lem:q-fact-tower}
		Let $X_0,\hdots, X_r$ be varieties of Fano type over some variety $V,$ and let $$X_0\xrightarrow{f_0}X_1\rightarrow\hdots\xrightarrow{f_{r-1}}X_r$$
		be a sequence of relative dimension one Mori fiber spaces over $V$. Let $(X_0,B_0)$ be a log Calabi--Yau pair with $B_0$ a $\qq$-divisor, and let $(X_i,B_i+\mathbf{M}_i)$ be the generalized log Calabi--Yau pairs induced by the canonical bundle formula. Then there exists a commutative diagram
		\[
		\xymatrix{ 
			(Y_0,C_0)\ar@{-->}[d]_-{\phi_0} \ar[r]^-{g_0} & (Y_1,C_1+\mathbf{N}_1)\ar@{-->}[d]^-{\phi_1} \ar[r] & \hdots \ar[r]^-{g_{r-1}} & (Y_r,C_r+\mathbf{N}_r) \ar@{-->}[d]^-{\phi_r}\\ 
			(X_0,B_0) \ar[r]^-{f_0} & (X_1,B_1+\mathbf{M}_1) \ar[r] & \hdots \ar[r]^-{f_{r-1}} & (X_r,B_r+\mathbf{M}_r)
		}
		\]
		of varieties over $V$ satisfying the following conditions:
		\begin{enumerate}
			\item each $\phi_i$ is crepant birational and extracts only generalized log canonical places of $(X_i,B_i+\mathbf{M}_i),$
			\item each $Y_i$ is $\qq$-factorial and of Fano type over $V$, 
			\item each $g_i$ is a Mori fiber space, and
			\item $g_i^{-1}(\lfloor C_{i+1}\rfloor)\subset \lfloor C_i\rfloor$ for each $0\leq i \leq r-1.$
		\end{enumerate}
	\end{lemma}
	\begin{proof}
		We prove, by induction on on $r$, the existence of a diagram over $V$ satisfying conditions (1)-(4) together with the additional condition:
		\begin{enumerate}
			\item[(5)] $\phi_r$ is a morphism.
		\end{enumerate}
		The base case $r=0$ follows from ~\cite[Corollary 1.37]{Kol13} and ~\cite[Lemma 2.17]{mor_con24}\par
		Assume the result for some $s\geq 0,$ and consider the case for $r=s+1.$ By assumption, there exists a commutative diagram 
		\[
		\xymatrix{ 
			(Y'_0,C'_0)\ar@{-->}[d]_-{\phi'_0} \ar[r]^-{g'_0} & (Y'_1,C'_1+\mathbf{N}'_1)\ar@{-->}[d]^-{\phi'_1} \ar[r] & \hdots \ar[r]^-{g'_{r-2}} & (Y'_{r-1},C'_{r-1}+\mathbf{N}'_{r-1}) \ar[d]^-{\phi'_{r-1}}\\ 
			(X_0,B_0) \ar[r]^-{f_0} & (X_1,B_1+\mathbf{M}_1) \ar[r] & \hdots \ar[r]^-{f_{r-2}} & (X_{r-1},B_{r-1}+\mathbf{M}_{r-1})
		}
		\]
		saatisfying conditions (1)-(5). By conditions (2) and (5), we may run a $K_{Y'_{r-1}}$-MMP over $X_r$ to obtain a commutative diagram
		\[
		\xymatrix{ 
			X_{r-1}\ar[d]_-{f_{r-1}} & Y'_{r-1}\ar[l]_-{\phi'_{r-1}}\ar[d] \ar@{-->}[r]^{\psi} & W \ar[d]^{h}\\ 
			X_r \ar[r]^-{{\rm Id}_{X_r}} & X_r & Z \ar[l]_{p}
		}
		\]
		satisfying the following conditions:
		\begin{enumerate}
			\item[(i)] $W$ and $Z$ are $\qq$-factorial varieties of Fano type over $V$,
			\item[(ii)] $\psi$ is a contracting birational map,
			\item[(iii)] $h$ is a Mori fiber space, and
			\item[(iv)] $p$ is a birational morphism.
		\end{enumerate}
		Here, we use the fact that $\dim Y'_{r-1}=\dim X_r+1$ to deduce that condition (iv) holds. \par
		Denote by $(W,C_W+\mathbf{N}_W)$ the generalized pair on $W$ pushed forward from $(Y'_{r-1},C'_1+\mathbf{N}'_1)$ and by $(Z,C_Z+\mathbf{N}_Z)$ the generalized pair induced from $(W,C_W+\mathbf{N}_W)$ via the canonical bundle formula. Thus, $p\colon (Z,C_Z+\mathbf{N}_Z) \rightarrow (X_{r},B_{r}+\mathbf{M}_{r})$ is crepant. Let $E$ be the reduced exceptional divisor of the birational morphism $p$. By condition (i), we may run an $E$-MMP over $X_r$ to obtain a commutative diagram 
		\[
		\xymatrix{  
			Z \ar[rd]_-{p} \ar@{-->}[rr]^-{\beta} & & Z' \ar[ld]^-{p'}\\
			& X_r &
		}
		\]
		where $\beta$ is a contracting birational map to a $\qq$-factorial variety $Z'$ and $p'$ is a small birational morphism. We set $Y_{r}=Z'$ and $\phi_r=p'.$ Note that $\phi_r$ satisfies condition (1) of the lemma since there are no $\phi_r'$-exceptional divisors, that it satisfies condition (5) by construction, and that $Y_r$ is of Fano type over $V$ by ~\cite[Lemma 2.17]{mor_con24}. Denote by $(Y_r,C_r+\mathbf{N}_r)$ the generalized pair on $Y_r$ pushed forward from $(Z,C_Z+\mathbf{N}_Z).$ By ~\cite[Lemma 2.12]{MM24}, we obtain commutative diagram 
		\[
		\xymatrix{ 
			(W,C_W+\mathbf{N}_W)\ar[d]_{h} & \ar@{-->}[l]_{\psi'} (Y_{r-1},C_{r-1}+\mathbf{N}_{r-1}) \ar[d]^{g_{r-1}}\\ 
			(Z,C_Z+\mathbf{N}_Z) & (Y_r,C_r+\mathbf{N}_r) \ar@{-->}[l]_{\beta^{-1}}
		}
		\]
		satisfying the following conditions:
		\begin{enumerate}
			\item[(a)] $\psi'$ is crepant birational and extracts only generalized log canonical places of $(W,C_W+\mathbf{N}_W),$
			\item[(b)] $Y_{r-1}$ is $\qq$-factorial,
			\item[(c)] $g_{r-1}$ is a Mori fiber space, and
			\item[(d)] $g_{r-1}^{-1}(\lfloor C_{r}\rfloor)\subset \lfloor C_{r-1}\rfloor.$
		\end{enumerate}
		Note that $Y_{r-1}$ is of Fano type over $V$ by condition (a) and ~\cite[Lemma 2.17]{mor_con24}. We have a commutative diagram
		\[
		\xymatrix{
			& & & (Y_{r-1},C_{r-1}+\mathbf{N}_{r-1}) \ar@{-->}[d]^-{\psi^{-1}\circ\psi'} \ar[r]^-{g_{r-1}} & (Y_{r},C_{r}+\mathbf{N}_{r}) \ar[dd]^-{\phi_{r}} \\
			(Y'_0,C'_0)\ar@{-->}[d]_-{\phi'_0} \ar[r]^-{g'_0} & (Y'_1,C'_1+\mathbf{N}'_1)\ar@{-->}[d]^-{\phi'_1} \ar[r] & \hdots \ar[r]^-{g'_{r-2}} & (Y'_{r-1},C'_{r-1}+\mathbf{N}'_{r-1}) \ar[d]^-{\phi'_{r-1}}\\ 
			(X_0,B_0) \ar[r]^-{f_0} & (X_1,B_1+\mathbf{M}_1) \ar[r] & \hdots \ar[r]^-{f_{r-2}} & (X_{r-1},B_{r-1}+\mathbf{M}_{r-1}) \ar[r]^-{f_{r-1}} & (X_{r},B_{r}+\mathbf{M}_{r}). 
		}
		\]
		Since $\psi^{-1}\circ \psi'$ is a crepant birational map extracting no generalized log canonical places of $(Y'_{r-1},C'_{r-1}+\mathbf{N}_{r-1}),$ we may apply ~\cite[Lemma 2.12]{MM24} repeatedly to obtain a commutative diagram 
		\[
		\xymatrix{
			(Y_0,C_0)\ar@{-->}[d]_-{\psi_0} \ar[r]^-{g_0} & (Y_1,C_1+\mathbf{N}_1)\ar@{-->}[d]^-{\psi_1} \ar[r] & \hdots \ar[r]^-{g_{r-2}} & (Y_{r-1},C_{r-1}+\mathbf{N}_{r-1}) \ar@{-->}[d]^-{\psi^{-1}\circ\psi'} \ar[r]^-{g_{r-1}} & (Y_{r},C_{r}+\mathbf{N}_{r}) \ar[dd]^-{\phi_{r}} \\
			(Y'_0,C'_0)\ar@{-->}[d]_-{\phi'_0} \ar[r]^-{g'_0} & (Y'_1,C'_1+\mathbf{N}'_1)\ar@{-->}[d]^-{\phi'_1} \ar[r] & \hdots \ar[r]^-{g'_{r-2}} & (Y'_{r-1},C'_{r-1}+\mathbf{N}'_{r-1}) \ar[d]^-{\phi'_{r-1}}\\ 
			(X_0,B_0) \ar[r]^-{f_0} & (X_1,B_1+\mathbf{M}_1) \ar[r] & \hdots \ar[r]^-{f_{r-2}} & (X_{r-1},B_{r-1}+\mathbf{M}_{r-1}) \ar[r]^-{f_{r-1}} & (X_{r},B_{r}+\mathbf{M}_{r}). 
		}
		\]
		satisfying the following conditions:
		\begin{enumerate}
			\item[(a')] each $\psi_i$ is crepant birational and extracts only generalized log canonical places of $(Y'_i,C'_i+\mathbf{N}'_i),$
			\item[(b')] each $Y_i$ is $\qq$-factorial,
			\item[(c')] each $g_i$ is a Mori fiber space, and
			\item[(d')] $g_i^{-1}(\lfloor C_{i+1}\rfloor)\subset \lfloor C_i \rfloor$ for each $0\leq i \leq r-1.$
		\end{enumerate}
		Each $Y_i$ is of Fano type over $V$ by condition (a') and ~\cite[Lemma 2.17]{mor_con24}. Setting $\phi_{r-1}=\phi'_{r-1}\circ \psi^{-1}\circ \psi'$ and $\phi_i=\phi'_i\circ \psi_i$ for $0\leq i \leq r-2,$ we complete the proof of the lemma.
	\end{proof}
	
	\subsection{Compound Du Val Singularities}
	In this short section, we recall the definitions of certain classes of singularities which will be important for us.
	\begin{definition}\label{def:du-val}
		{\em
			Let $(X,x_0)$ be a germ of a normal surface. We say that $x_0$ is a \textit{Du Val singularity} if $X$ is singular and canonical at $x_0$.
		}
	\end{definition}
	Du Val singularities can be separated into subclasses of \textit{A-type}, \textit{D-type} and \textit{E-type} singularities according to the configurations of their dual graphs (see ~\cite[Theorem 4.22]{KM98}).
	\begin{definition}\label{def:cdv-sing}
		{\em
			Let $(X,x_0)$ be a germ of a normal $n$-fold for some $n\geq 2.$ We say that $x_0$ is a \textit{compound Du Val} (or \textit{cDV}) singularity if one of the following holds:
			\begin{enumerate}
				\item $n=2$ and $x_0$ is a Du Val singularity, or
				\item $n\geq 3$, $X$ is singular at $x_0$ and there is an effective Cartier divisor $x_0\in S\subset X$ such that $(S,x_0)$ is a compound Du Val singularity (in particular, we assume that such $S$ is normal at $x_0$).
			\end{enumerate}
			More generally, we say that $x_0$ is \textit{at worst a compound Du Val} singularity if $x_0$ is smooth or a compound Du Val singularity.
		}
	\end{definition}
	
	Since Du Val singularities are hypersurface singularities (see ~\cite[Theorem 4.22]{KM98}), it follows from the definition that a cDV singularity must be a hypersurface singularity. We will be interested in the following subclass of cDV singularities.
	
	\begin{definition}[c.f. {~\cite[1.42]{Kol13}}]\label{def:cA-sing}
		{\em
			Let $(X,x_0)$ be the germ of a normal $n$-fold hypersurface singularity for some $n\geq 2.$ We say that $x_0$ is a \textit{compound A-type} (or \textit{cA-type}) singularity if $x_0$ is formally isomorphic to $\kk\llbracket x_1,\hdots, x_{n+1}\rrbracket/(f),$ where the quadratic part of $f$ has rank at least two.
		}
	\end{definition}
	
	Up to formal change of coordinates, one can arrange for the defining equation $f$ of a $n$-fold cA-type singularity to have the form $f(x_1,\hdots, x_{n+1})=x_1x_2-g(x_3,\hdots, x_{n+1}).$ We remark that a singularity is a cA-type singularity in this sense if and only if it is a cDV singularity such that a general surface section has a A-type Du Val singularity.
	\subsection{Cox Rings and Mori Dream Spaces}\label{subsec:cox-rings}
	
	In this subsection we recall the definitions and prove some lemmas about Cox rings and Mori dream spaces.  
	
	\begin{notation}\label{not:cox-sheaf}
		{\em
			Let $X$ be a normal projective variety with finitely generated divisor class group ${\rm Cl}(X)$. We will denote by $\mathcal{C}ox(X)$ the \textit{Cox sheaf} of $X$ as in ~\cite[Construction 1.4.2.1]{ADHL15}. It is a quasi-coherent sheaf of ${\rm Cl}(X)$-graded commutative $\mathcal{O}_X$-algebras. We will denote by ${\rm Cox}(X)=\Gamma(X,\mathcal{C}ox(X))$ the \textit{Cox ring} of $X$. It is a ${\rm Cl}(X)$-graded commutative $\kk$-algebra. Given $[D]\in {\rm Cl}(X)$ we will denote by $\mathcal{C}ox(X)_{[D]}$ and ${\rm Cox}(X)_{[D]}$ the corresponding graded pieces. We will denote by $\widehat{X}=\mathcal{S}pec_X\mathcal{C}ox(X)$ the \textit{characteristic space} of $X$ and by $\overline{X}=\Spec {\rm Cox}(X)$ the \textit{total coordinate space} of $X.$ We will denote by $H_X=\Spec \kk[{\rm Cl}(X)]$ and refer to it as the \textit{characteristic quasitorus} of $X.$
		}
	\end{notation}

	\begin{notation}\label{not:cox-actions}
		{\em
			Let $X$ be a normal projective variety with finitely generated divisor class group ${\rm Cl}(X)$, and assume that $\mathcal{C}ox(X)$ is finitely generated over $\kk.$ We note that $\widehat{X}$ and $\overline{X}$ admit $H_X$-actions corresponding to the ${\rm Cl}(X)$-gradings of $\mathcal{C}ox(X)$ and ${\rm Cox}(X).$ We have an $H_X$-invariant projection $p_X\colon \widehat{X}\rightarrow X$ correpsonding to the inclusion $\mathcal{O}_X\hookrightarrow \mathcal{C}ox(X)$ as the degree-$0$ piece, and we have an $H_X$-equivariant open immersion $i_X\colon \widehat{X}\hookrightarrow \overline{X}$ corresponding to the identity morphism on global sections.
		}
	\end{notation}
	
	\begin{notation}\label{not:div-alg}
		{\em
			Let $X$ be a normal projective variety and let $K\subset {\rm WDiv}(X)$ be a finitely generated subgroup of (integral) Weil divisors on $X.$ We will denote by 
			$$\mathcal{R}(X,K)=\bigoplus_{D\in K}\mathcal{O}_X(D)$$ the \textit{sheaf of divisorial algebras} associated with $K$ and by
			$$R(X,K)=\Gamma(X,\mathcal{R}(X,K))=\bigoplus_{D\in K}\Gamma(X,\mathcal{O}_X(D))$$
			the \textit{multi-section ring} associated with $K.$ We note that $\mathcal{R}(X,K)$ is a quasi-coherent sheaf of $K$-graded commutative $\mathcal{O}_X$-algebras and that $R(X,K)$ is a $K$-graded commutative $\kk$-algebra. In both cases, multiplication is defined by multiplying homogeneous elements inside the function field $\kk(X).$ 
		}
	\end{notation}
	
	We recall the following definition from ~\cite{HK00}.
	\begin{definition}\label{def:mds}
		Let $X$ be a normal projective variety. We say that $X$ is a \textit{Mori dream space} if the following conditions are satisfied:
		\begin{enumerate}
			\item $X$ is $\qq$-factorial and ${\rm Pic}(X)_\qq=N^1(X)_\qq,$
			\item ${\rm Nef}(X)$ is generated as a cone by finitely many semi-ample line bundles, and
			\item there is a finite collection of small $\qq$-factorial modifications $f_i\colon X_i\dashrightarrow X,$ $1\leq i \leq n$ such that ${\rm Mov}(X)=\bigcup_{i=1}^nf_i^*{\rm Nef}(X_i)$ and such that each $X_i$ satisfies conditions (1) and (2).
		\end{enumerate}
	\end{definition}
	
	The following well-known lemma establishes the fundamental connection between Cox rings and Mori dream spaces.
	
	\begin{lemma}\label{lem:mds-ft-cox}
		Let $X$ be a normal, $\qq$-factorial projective variety with finitely generated divisor class group ${\rm Cl}(X)$ such that ${\rm Cl}(X)_\qq={\rm Pic}(X)_\qq=N^1(X).$ Let $K\subset {\rm WDiv}(X)$ be a subgroup of Weil divisors such that $K_\qq\xrightarrow{\cong}{\rm Cl}(X)_\qq$ is an isomorphism. Then the following hold:
		\begin{enumerate}
			\item ${\rm Cox}(X)$ is a finitely generated $\kk$-algebra if and only if $R(X,K)$ is a finitely generated $\kk$-algebra,
			\item $X$ is a Mori dream space if and only if ${\rm Cox}(X)$ is a finitely generated $\kk$-algebra.
		\end{enumerate}
	\end{lemma}
	\begin{proof}
		Item (1) follows from ~\cite[Remark 2.19]{GOST15} and the proof of ~\cite[Lemma 2.21]{GOST15}. In light of (1), item (2) follows from ~\cite[Proposition 2.9]{HK00}. 
	\end{proof}

	We recall the following useful sufficient criterion for a variety to be a Mori dream space.
	\begin{lemma}\label{lem:mds-fano}
		Let $X$ be a $\qq$-factorial Fano type variety. Then $X$ is a Mori dream space with a finitely generated divisor class group.
	\end{lemma}
	\begin{proof}
		This follows from Lemma ~\ref{lem:mds-ft-cox}, ~\cite[Corollary 1.3.2]{BCHM10} and ~\cite[Theorem 3.1]{Tot12}.
	\end{proof}
	
	For use in the proof of Theorem ~\ref{introthm:abs-comp-is-min}, we introduce the following ad hoc definition.
	
	\begin{definition}\label{def:m-mov}
		{\em
			Let $X$ be a Mori dream space and let $D$ be a movable Weil divisor on $X.$ Given an integer $m>0$ we will say that $D$ is \textit{$m$-very movable} if there is a small $\qq$-factorial modification $X'$ of $X$ on which the strict transform $D'$ of $D$ satisfies $D'\sim  \sum_{i=1}^mB_i$ for basepoint-free Cartier divisors $B_1,\hdots, B_m$ on $X'$, none of which are linearly equivalent to zero.
		}
	\end{definition}
	
	\begin{lemma}\label{lem:fin-m-mov}
		Let $X$ be a Fano type variety and let $m>0$ be an integer. Then there are only finitely many numerical classes of movable Cartier divisors on $X$ that are not $m$-very movable. 
	\end{lemma}
	\begin{proof}
		It suffices to show that there are only finitely many numerical classes of nef Cartier divisors on $X$ that are not $m$-very movable. Since $X$ is a Mori dream space, its nef cone ${\rm Nef}(X)$ is a strongly convex rational polyhedral cone of full dimension in $N^1(X)_\rr={\rm Pic}(X)_\rr.$ We note that ${\rm Pic}(X)$ is torsion-free by ~\cite[Corollary 1]{Zha06}. It follows that the semi-group $S={\rm Nef}(X)\cap {\rm Pic}(X)$ is finitely generated, say by $L_1,\hdots, L_r\in S.$ It follows from the Basepoint-free Theorem ~\cite[Theorem 3.3]{KM98} that there exists an integer $n_0>0$ such that $nL_i$ is basepoint-free for each $1\leq i \leq r$ and all $n\geq n_0.$ Let $\phi\colon \nn^r\rightarrow {\rm Pic}(X)$ be the homomorphism $\phi(a_1,\hdots, a_r)=\sum_{i=1}^ra_iL_i.$ Then any nef Cartier divisor that is not in $\phi\left(\{0,\hdots, n_0m\}^r\right)$ is $m$-very movable.
	\end{proof}
	
	\begin{notation}\label{not:vertex}
		{\em
			Let $X$ be a Mori dream space with finitely generated divisor class group ${\rm Cl}(X).$ We will refer to the closed point in $\overline{X}$ corresponding to the maximal ideal $\bigoplus_{[D]\neq 0}{\rm Cox}(X)_{[D]}$ as the \textit{vertex} of $\overline{X}.$
		}
	\end{notation}
	
	The following is well-known to experts. We include a proof for completeness.
	\begin{lemma}\label{lem:spreading-out}
		Let $X$ be a Mori dream space with finitely generated divisor class group ${\rm Cl}(X).$ If the vertex of $\overline{X}$ has embedding dimension $d$, then $\overline{X}$ admits an $H_X$-equivariant closed embedding into $\mathbb{A}^d.$
	\end{lemma}
	\begin{proof}
		Write $M=\bigoplus_{[D]\neq 0}{\rm Cox}(X)_{[D]}.$ By assumption, $M/M^2$ is a $d$-dimensional $\kk$-vector space. We may choose ${\rm Cl}(X)$-homogeneous $f_1,\hdots, f_d\in M$ whose images generate $M/M^2.$ Then for each $n\geq 1,$ $M^n/M^{n+1}$ is generated as a $\kk$-vector space by the collection of degree $n$ monomials in $f_1,\hdots, f_d.$ Since the ideals $M^n$ are ${\rm Cl}(X)$-homogeneous for all $n\geq 0,$ the quotients $M^n/M^{n+1}$ inherit a ${\rm Cl}(X)$-grading for all $n\geq 0.$ Given $[D]\in {\rm Cl}(X),$ it follows from the finite-dimensionality of ${\rm Cox}(X)_{[D]}$ over $\kk$ and the Krull Intersection Theorem that there are only finitely many $n\geq 0$ such that ${\rm Cox}(X)_{[D]}\cap M^n\neq 0.$ \par
		We claim that the $\kk$-algebra homomorphism $$\phi\colon \kk[x_1,\hdots, x_d]\rightarrow {\rm Cox}(X)$$ defined by $\phi(x_i)=f_i$ is surjective. Let $g\in {\rm Cox}(X)$ be nonconstant and homogeneous, say of degree $[D]\in {\rm Cl}(X).$ Let $n\geq 1$ be such that $g\in M^n\setminus M^{n+1}.$ It follows that there exists a polynomial $p\in \kk[x_1,\hdots, x_d]_n$ such that $p(f_1,\hdots, f_d)\in {\rm Cox}(X)_{[D]}\cap M^n$ and $g\equiv p(f_1,\hdots,f_d)$ modulo $M^{n+1}.$ Thus, $g-p(f_1,\hdots, f_d)\in {\rm Cox}(X)_{[D]}\cap M^{n+1}.$ Continuing in this manner, we obtain after finitely many steps a polynomial $q\in \kk[x_1,\hdots,x_d]$ satisfying $g=q(f_1,\hdots, f_d),$ i.e. such that $g=\phi(q).$ The surjectivity of $\phi$ follows.\par
		The surjection $\phi$ defines a closed embedding $\overline{X}\hookrightarrow \mathbb{A}^d.$ We can equip $\kk[x_1,\hdots, x_d]$ with a ${\rm Cl}(X)$-grading, and hence equip $\mathbb{A}^d$ with an $H_X$-action, by assigning ${\rm deg}(x_i)={\rm deg}(f_i).$ The homomoprhism $\phi$ is ${\rm Cl}(X)$-homogeneous, hence the embedding $\overline{X}\hookrightarrow\mathbb{A}^d$ is $H_X$-equivariant.
	\end{proof}
	
	\begin{lemma}\label{lem:cA-cox}
		Let $X$ be a Mori dream space with finitely generated divisor class group ${\rm Cl}(X),$ and write $d=\dim\overline{X}.$ If the vertex of $\overline{X}$ is a $cA$-type singularity, then $${\rm Cox}(X)=\kk[x_1,\hdots,x_{d+1}]/f,$$
		where $x_1,\hdots, x_{d+1}$ are ${\rm Cl}(X)$-homogeneous and $f$ is a ${\rm Cl}(X)$-homogeneous polynomial of one of the following forms:
		\begin{enumerate}
			\item $f(x_1,\hdots, x_{d+1})=x_1x_2-g(x_3,\hdots, x_{d+1}),$
			\item $f(x_1,\hdots, x_{d+1})=\sum_{i=1}^rx_i^2-g(x_1,\hdots,x_{d+1}),$ where $2\leq r \leq d+1,$ ${\rm deg}(x_i)-{\rm deg}(x_j)$ is a nonzero $2$-torsion element of ${\rm Cl}(X)$ for each $1\leq i\neq j\leq r$, and all monomials appearing in $g$ are of total degree at least $3.$
		\end{enumerate}
	\end{lemma}
	\begin{proof}
		Since cA-type singularities are hypersurface singularities, it follows from Lemma ~\ref{lem:spreading-out} that $${\rm Cox}(X)=\kk[x_1,\hdots,x_{d+1}]/f$$ with $x_1,\hdots, x_{d+1}$ and $f$ ${\rm Cl}(X)$-homogeneous. It follows from the definition of cA-type singularities that $f$ has no linear term and that its quadratic term has rank at least $2.$ Applying ~\cite[Proposition 2.1]{FH20} to the quadratic part of $f$, we may assume that $f$ is of the form $$f(x_1,\hdots. x_{d+1})=\sum_{i=1}^rx_{2i-1}x_{2i}+\sum_{j=1}^sx_{2r+j}^2+h(x_1,\hdots, x_{d+1}),$$ where $r,s\geq 0$ satisfy $2r+s\geq 2$ and all monomials appearing in $h$ are of total degree at least $3.$ For each $2r+1\leq i \neq j \leq 2r+s$, we have that ${\rm deg}(x_i^2)={\rm deg}(x_j^2)$ and hence that ${\rm deg}(x_i)-{\rm deg}(x_j)$ is a $2$-torsion element of ${\rm Cl}(X).$ If $r=0$ and ${\rm deg}(x_i)-{\rm deg}(x_j)$ is nonzero for each $2r+1\leq i \neq j \leq 2r+s,$ then we set $g=-h$ and are done.\par
		From now on, we assume that $r\neq 0$ or that ${\rm deg}(x_i)-{\rm deg}(x_j)=0$ for some $2r+1\leq i \neq j \leq 2r+s.$ If ${\rm deg}(x_i)-{\rm deg}(x_j)=0$ for some $2r+1\leq i \neq j \leq 2r+s,$ then we may make the ${\rm Cl}(X)$-homogeneous change of variables 
		\[ 
		y_k=\begin{cases} 
			x_k & k\neq i,j \\
			x_i+\sqrt{-1}x_j & k=i \\
			x_i-\sqrt{-1}x_j & k=j. 
		\end{cases}
		\]
		We obtain $x_i^2+x_j^2=y_iy_j,$ hence we may assume that $r>0.$ Choose a $\zz$-basis $u_1,\hdots, u_\rho$ for ${\rm Cl(X)}/{\rm Cl(X)}_{\rm tors}$ such that ${\rm Psef}(X)\subset {\rm Cone}(u_1,\hdots, u_\rho),$ and define a homomorphism $$\psi\colon {\rm Cl}(X)\rightarrow \zz$$ by $\psi(u_i)=1$ for $1\leq i \leq \rho.$ Denoting by $\phi$ the composite $${\rm Cl}(X)\rightarrow {\rm Cl}(X)/{\rm Cl}(X)_{\rm tors}\xrightarrow{\psi}\zz,$$
		we obtain a nonnegative $\zz$-grading on $\kk[x_1,\hdots, x_{d+1}]$ as a coarsening the ${\rm Cl}(X)$-grading. Reindexing if necessary, we may assume that ${\rm deg}(x_2)\leq {\rm deg}(x_1)$ with respect to this $\zz$-grading.\par
		Collecting all terms of $f$ in which $x_1$ appears, we have $$f(x_1,\hdots,x_{d+1})=x_1(x_2+h_0(x_1,\hdots, x_{d+1}))+h_1(x_2,\hdots,x_{d+1})$$ for some ${\rm Cl}(X)$-homogeneous polynomials $h_0\in \kk[x_1,\hdots,x_{d+1}]$ and $h_1\in \kk[x_2,\hdots, x_{d+1}].$ While $h_0$ a priori depends on $x_1$ and $x_2$, we note that it must actually depend on neither $x_1$ nor $x_2.$ Indeed, it follows from the fact that all monomials in $h$ have total degree at least $3$ that all monomials in $h_0$ have total degree at least $2.$ Since ${\rm deg}(x_2)\leq {\rm deg}(x_i)$ with respect to the coarsened $\zz$-grading for $i=1,2$, $x_2$ would have degree strictly smaller than any term of $h_0$ in which $x_1$ or $x_2$ appears. This would violate ${\rm Cl}(X)$-homogeneity.\par
		We make the ${\rm Cl}(X)$-homogeneous change of variables
		\[ 
		y_k=\begin{cases} 
			x_k & k\neq 2 \\
			x_2+h_0(x_3,\hdots,x_{d+1}) & k=2.
		\end{cases}
		\]
		We obtain an expression $$f(y_1,\hdots, y_{d+1})=y_1y_2+h_1(y_2+h_0(y_3,\hdots,y_{d+1}),y_3,\hdots, y_{d+1}).$$
		Thus, we may assume that, in our original notation, $h$ does not depend on $x_1:$
		$$f(x_1,\hdots, x_{d+1})=x_1x_2+h(x_2,\hdots,x_{d+1}).$$
		Collecting all terms of $f$ in which $x_2$ appears, we have $$f(x_1,\hdots,x_{d+1})=x_2(x_1+h_0(x_2,\hdots,x_{d+1}))+h_1(x_3,\hdots,x_{d+1}).$$
		We make the ${\rm Cl}(X)$-homogeneous change of variables
		\[ 
		y_k=\begin{cases} 
			x_k & k\neq 1 \\
			x_1+h_0(x_2,\hdots,x_{d+1}) & k=1.
		\end{cases}
		\]
		We obtain an expression 
		$$f(y_1,\hdots, y_{d+1})=y_1y_2+h_1(y_3,\hdots, y_{d+1}).$$
		Setting $g=-h_1,$ we are done.
	\end{proof}
	
	\begin{notation}\label{not:[D]-div}
		{\em
			Let $X$ be a Mori dream space with finitely generated divisor class group. Let $f\in {\rm Cox}(X)$ be homogeneous of degree $[D]\in {\rm Cl}(X).$ We will denote by ${\rm div}_{\overline{X}}(f)\in {\rm WDiv}(\overline{X})$ the divisor of zeros of $f$ viewed as a regular function on $\overline{X},$ and we will denote by ${\rm div}_X(f)\in {\rm WDiv}(X)$ the divisor ${\rm div}_{[D]}(f)$ on $X$ constructed in ~\cite[Construction 1.5.2.1]{ADHL15}.
		}
	\end{notation}
	
	We record the following two results comparing singularities of a Mori dream space with that of its total coordinate space.
	
	\begin{lemma}\label{lem:check-sings-cox}
		Let $X$ be a Mori dream space with finitely generated divisor class group. Let $f_1,\hdots, f_r\in {\rm Cox}(X)$ be homogeneous. Given a boundary $\overline{B}=\sum_{i=1}^ra_i{\rm div}_{\overline{X}}(f_i)$ on $\overline{X},$ consider the boundary $B=\sum_{i=1}^ra_i{\rm div}_{X}(f_i)$ on $X$. If $(\overline{X},\overline{B})$ is klt (resp. lc), then $(X,B)$ is klt (resp. lc).
	\end{lemma}
	\begin{proof}
		This follows from ~\cite[Remark 2.19]{GOST15}, the proof of ~\cite[Lemma 2.21]{GOST15} and Step 2 of the proof of ~\cite[Theorem 3.2]{KO15}.
	\end{proof}
	
	\begin{lemma}\label{lem:induced-cox-pair}
		Let $X$ be a Mori dream space with finitely generated divisor class group. Let $f_1,\hdots, f_r\in {\rm Cox}(X)$ be homogeneous. Given a boundary $B=\sum_{i=1}^ra_i{\rm div}_{X}(f_i)$ on $X,$ consider the boundary $\overline{B}=\sum_{i=1}^ra_i{\rm div}_{\overline{X}}(f_i)$ on $\overline{X}$. If $(X,B)$ is log Calabi--Yau, then $(\overline{X},\overline{B})$ is log canonical.
	\end{lemma}
	\begin{proof}
		This follows from ~\cite[Remark 2.19]{GOST15}, the proof of ~\cite[Lemma 2.21]{GOST15} and the proof of ~\cite[Theorem 3.1]{KO15}. 
	\end{proof}
	
	\section{Complexity One Pairs}
	In this section we prove Theorems \ref{introthm:compl-less-than-2}, \ref{introthm:abs-compl-cA}, \ref{introthm:abs-compl-3/2} and \ref{introthm:abs-comp-is-min}. First, we prove Theorem~\ref{introthm:compl-less-than-2} using ideas from \cite{BMSZ18}.
	
	\begin{proof}[Proof of Theorem~\ref{introthm:compl-less-than-2}]
		To show that $X$ is Fano type, we imitate the proof of ~\cite[Lemma 5.3]{BMSZ18}, explaining how the situation differs when considering the complexity $c(X,B)$ instead of the fine complexity $\overline{c}(X,B).$ By ~\cite[Theorem 5.1]{FG12}, it suffices to show that there exists a birational morphism $Y\rightarrow X$ from some Fano type variety $Y.$ So, by Lemma ~\ref{lem:exist-dlt-mod} and Corollary ~\ref{cor:compl-dlt-mod}, we may assume that $X$ is $\qq$-factorial and $(X,B)$ is dlt. It follows from the proof of ~\cite[Corollary 4.2]{BMSZ18} that there exists a big and semi-ample divisor $A\geq 0$ and an effective divisor $B'$ such that $(X,B'+A)$ is klt and $K_X+B'+A\sim_\rr N\geq 0$ has numerical dimension $0$. By ~\cite[Lemma 5.1]{BMSZ18}, we may assume that $N$ shares no components with $B$ and that $N$ contains no log canonical centers of $(X,B)$.

		If $N=0,$ then $(X,B'+A)$ is a klt log Calabi--Yau pair with $-K_X\sim_{\rr}B'+A$ big. In this case, it follows from ~\cite[Lemma 2.5.3]{BMSZ18} that $X$ is Fano type. If $N\neq 0,$ we use ~\cite[Theorem 1.2]{BCHM10} to obtain a log terminal model $\pi\colon X \dashrightarrow Y$ for $(X,B'+A).$ Write $B_Y=\pi_*B,$ $B'_Y=\pi_*B'$ and $A_Y=\pi_*A.$ Then $(Y,B'_Y+A_Y)$ is a klt log Calabi--Yau pair with $-K_Y\sim_{\rr}B'_Y+A_Y$ big, hence $Y$ is Fano type. The divisor $-(K_Y+B_Y)$ is movable since it is pushed forward from the nef divisor $-(K_X+B).$ Since $Y$ is Fano type, we may replace $Y$ with a small $\qq$-factorial modification on which $-(K_Y+B_Y)$ is nef. The divisors contracted by $\pi$ are precisely the components of $N$. Since $N$ is nonzero and shares no components with $B,$ it follows from Lemma ~\ref{lem:complexity-and-exceptional-divisors} and ~\cite[Corollary 1.3]{BMSZ18} that $N$ has exactly one component $E$ and that $c(Y,B_Y)<1.$ From here, the proof of ~\cite[Lemma 5.3]{BMSZ18} applies verbatim to show that we can find a big and nef divisor $A'\geq 0$ and an effective divisor $B''$ such that $(X,B''+A')$ is a klt log Calabi--Yau pair. As before, it follows from ~\cite[Lemma 2.5.3]{BMSZ18} that $X$ is Fano type.\par
		To see that $\Spec {\rm Cox}(X)$ has at worst a compound Du Val singularity at its vertex, we use the fact that $X$ is Fano type and $-(K_X+B)$ is nef to choose $0\leq D \sim_\rr-(K_X+B)$ such that $(X,B+D)$ is log Calabi--Yau. Note that $c(X,B+D)\leq c(X,B)<2.$ The desired conclusion now follows from Lemma ~\ref{lem:induced-cox-pair} applied to $(X,B+D)$ and the proof of ~\cite[Lemma 2.4.3]{BMSZ18}.\par
		From now on, we assume that $(X,B)$ is log Calabi--Yau. As above, we may assume that $X$ is $\qq$-factorial. Since $c(X,B)<2,$ it follows that $|B|>\dim X$ if $\rho(X)>1.$ Thus, it follows from ~\cite[Lemma 3.3]{BMSZ18} that we must have $\rho(X)=1$ if the components of $B$ all span the same ray of the cone of effective divisors. Having made this observation, we can now apply the proof of ~\cite[Theorem 3.2]{BMSZ18} verbatim to conclude that the components of $B$ generate ${\rm Cl}(X)_\qq.$
	\end{proof}
	
	Next, we use the theory of Cox rings (Subsection \ref{subsec:cox-rings}) to prove Theorem \ref{introthm:abs-compl-cA}.
	
	\begin{proof}[Proof of Theorem~\ref{introthm:abs-compl-cA}]
		
		Condition (1) clearly implies condition (2). To show that condition (2) implies condition (3), let us assume that $\widehat{c}(X) \in [1, 3/2)$. Since $\widehat{c}(X)\geq 1,$ it follows from Theorem~\ref{introthm:BMSZ18}, that $X$ is not toric.
		By the definition of $\hat{c}$, we may choose a log Calabi-Yau pair $(X, B)$ such that $c(X, B) \in (1, 3/2)$. It now follows from Theorem~\ref{introthm:compl-less-than-2} that $X$ is of Fano type. Lemma ~\ref{lem:induced-cox-pair} applied to $(X,B)$ and ~\cite[Lemma 2.4.3]{BMSZ18} show that $\Spec {\rm Cox}(X)$ has either a smooth point or a cA-type singularity at its vertex. Since $X$ is not toric, it follows from ~\cite[Theorem 2.5.4]{BMSZ18} that $\Spec {\rm Cox}(X)$ cannot be smooth at its vertex.\par 
		In the remaining part of this proof, we show that (3) implies (1), and we also prove the last claim of Theorem~\ref{introthm:abs-compl-cA}. From now on, we assume that $X$ is a non-toric variety of Fano type and that $\overline{X}=\Spec {\rm Cox}(X)$ has a cA-type singularity at its vertex. We may assume that $X$ is $\qq$-factorial. Indeed, if $X'\rightarrow X$ is a small $\qq$-factorial modification then we will have $1 \leq \widehat{c}(X)\leq \widehat{c}(X')$ by Lemma ~\ref{lem:abs-comp-bir-contr} and will have ${\rm Cox}(X')={\rm Cox}(X)$. We may express ${\rm Cox}(X)=\kk[x_1,\hdots,x_{d+1}]/f$ as in Lemma ~\ref{lem:cA-cox}. Writing $\overline{T}=\mathbb{A}^{d+1},$ this expression provides us with an $H_X$-equivariant closed embedding $\overline{j}\colon \overline{X}\hookrightarrow\overline{T}.$ For $1\leq i\leq d+1,$ denote by $\overline{D}_i={\rm div}_{\overline{T}}(x_i).$ It follows from inversion of adjunction that the pair $(\overline{T},\overline{X}+\sum_{i=3}^{d+1}\overline{D}_i)$ is log canonical in a neighborhood of the origin. Indeed, the pair obtained from restricting to $\bigcap_{i=3}^{d+1}\overline{D}_i\cong \mathbb{A}^2$ is $(\mathbb{A}^2,\{f(x_1,x_2,0,\hdots,0)=0\})$. It follows from the form of $f$ as described in Lemma ~\ref{lem:cA-cox} that $\{f(x_1,x_2,0,\hdots,0)=0\}$ is a reduced nodal curve with two branches at the origin, showing that the pair $(\mathbb{A}^2,\{f(x_1,x_2,0,\hdots,0)=0\})$ is log canonical at the origin. Since the boundary divisor $\overline{X}+\sum_{i=3}^{d+1}\overline{D}_i$ is $H_X$-invariant and since any neighborhood of the origin meets every $H_X$-orbit in $\overline{T},$ it follows that $(\overline{T},\overline{X}+\sum_{i=3}^{d+1}\overline{D}_i)$ is log canonical. \par
		As in ~\cite[Construction 3.2.5.7]{ADHL15}, there is a commutative diagram 
		\[
		\xymatrix{ 
			\overline{X} \ar@{^{(}->}[r]^-{\overline{j}}&\overline{T}\\ 
			\widehat{X}\ar[u]^-{i_X} \ar[d]_-{p_X} \ar@{^{(}->}[r]^-{\widehat{j}} & \widehat{T}\ar[u]_-{i_T}\ar[d]^-{p_T} \\ 
			X \ar@{^{(}->}[r]^-{j} &
			T
		}
		\]
		where $T$ is a projective toric variety containing $X$ as a hypersurface such that $X={\rm div}_{T}(f).$ For each $1\leq i \leq d+1,$ write $D_i={\rm div}_{T}(x_i).$ It follows from Lemma ~\ref{lem:check-sings-cox} and the work of the previous paragraph that $(T,X+\sum_{i=3}^{d+1}D_i)$ is log canonical. Since $f$ contains a monomial of the form $x_1x_2$ or $x_1^2+x_2^2$ by Lemma ~\ref{lem:cA-cox}, it follows that $2X\sim 2(D_1+D_2)$. But $\sum_{i=1}^{d+1}D_i$ is the toric boundary on $T,$ so $2(K_T+X+\sum_{i=3}^{d+1}D_i)\sim 0.$ Performing adjunction to $X,$ we obtain a log Calabi--Yau pair $(X,B)$ of index at most $2$. We note that $B$ contains at least $d-1$ components with coefficient $1.$ Since $d=\dim X + \rho(X),$ it follows that $c(X,B)\leq 1.$ Since $X$ is not toric, it follows that $c(X,B)=1$ and that $B$ is a reduced boundary.
	\end{proof}
	
	In light of Theorem \ref{introthm:abs-compl-cA}, it is interesting to study the set of possible absolute complexities.
	
	\begin{question}\label{ques:other-vals-in-[0,2)}
		What are the other possible values of $\widehat{c}$ in the interval $[0,2)?$ Is the set of possible absolute complexities discrete?
	\end{question}

	To prove Theorem \ref{introthm:abs-compl-3/2}, we construct examples in every dimension $\geq 3$ with absolute complexity $\frac{3}{2}$.
	
	\begin{proof}[Proof of Theorem~\ref{introthm:abs-compl-3/2}]
		Fix $n\geq 3.$ Define $$R=\kk[x,y,z,w_1,\hdots, w_{n-1}]/\left(x^2-y^2z+4z^3+y^2w_1+\sum_{i=2}^{n-1}w_i^6\right),$$ and write $\overline{X}=\Spec R.$
		We claim that $R$ is a UFD. It suffices to check that $R$ is a normal domain and that ${\rm Cl}(\overline{X})=0.$ That $R$ is a domain follows from the fact that the polynomial $y^2z-4z^3-y^2w_1-\sum_{i=2}^{n-1}w_i^6$ cannot be a square in $\kk[y,z,w_1,\hdots,w_{n-1}],$ and normality follows from Serre's criterion as $\overline{X}$ is easily seen to be nonsingular in codimension one. Write $E=\{y=0\}\subset \overline{X}.$ We note that $E$ is a prime divisor since $R/(y)=\kk[x,z,w_1,\hdots,w_{n-1}]/(x^2+4z^3+\sum_{i=2}^{n-1}w_i^6)$ is an integral domain. We have an exact sequence $$\zz\rightarrow {\rm Cl}(\overline{X})\rightarrow {\rm Cl}(\overline{X}\setminus E)\rightarrow 0,$$
		where the leftmost arrow maps $n\mapsto [nE].$ Since $E$ is principal, it follows that this leftmost map is zero. On the other hand, $\overline{X}\setminus E=\Spec R[y^{-1}].$ But $R[y^{-1}]=\kk[x,y^{\pm},z,w_2,\hdots,w_{n-1}]$ is a UFD, so ${\rm Cl}(\overline{X}\setminus E) = 0.$ It follows that ${\rm Cl}(\overline{X})=0,$ as desired.\par
		We equip $R$ with a nonnegative $\zz$-grading by declaring ${\rm deg}(x)=3$, ${\rm deg}(y)={\rm deg}(z)={\rm deg}(w_1)=2$ and ${\rm deg}(w_{2})=\hdots={\rm deg}(w_{n-1})=1.$ Define $X=\text{Proj } R.$ Thus, $X$ is a normal projective $n$-fold. We note that each of the classes $x,y,z,w_1,\hdots, w_{n-1}$ are prime in $R.$ This is easy to see for each of the variables $y,z,w_1,\hdots, w_{n-1}.$ For the variable $x,$ we note that the class of $y$ is not a zero-divisor in the quotient $R/(x)=\kk[y,z,w_1,\hdots, w_{n-1}]/(-y^2z+4z^3+y^2w_1+\sum_{i=2}^{n-1}w_i^6)$ and that $R/(x)[y^{-1}]=\kk[y^\pm,z,w_2,\hdots,w_{n-1}]$ is a domain. It now follows from ~\cite[Theorem 3.2.1.4]{ADHL15} that ${\rm Cl}(X)=\zz$ and that $R={\rm Cox}(X)$. We note that $X$ is $\qq$-factorial since ${\rm Cl}(X)=\zz.$
		
		Since the quadratic part of $x^2-y^2z+4z^3+y^2w_1+\sum_{i=2}^{n-1}w_i^6$ only has rank one, the vertex of $\overline{X}$ cannot be a cA-type singularity. It follows from Theorem ~\ref{introthm:abs-compl-cA} that $\widehat{c}(X)\geq \frac{3}{2}.$ 
		Define $$B=\frac{1}{2}{\rm div}_X(z)+\sum_{i=1}^{n-1}{\rm div}_X(w_i).$$ We claim that $(X,B)$ is a log Calabi--Yau pair. Since $c(X,B)=\frac{3}{2},$ this would imply that $\widehat{c}(X)=\frac{3}{2}$ and hence that $X$ is Fano type by Theorem ~\ref{introthm:compl-less-than-2}. Since ${\rm Cl}(X)=\zz,$ it follows that $X$ is Fano.
		
		It follows from ~\cite[Proposition 3.3.3.2]{ADHL15} that $2(K_X+B)\sim 0,$ so we must simply show that $(X,B)$ is log canonical. Writing $$\overline{B}=\frac{1}{2}{\rm div}_{\overline{X}}(z)+\sum_{i=1}^{n-1}{\rm div}_{\overline{X}}(w_i),$$
		it suffices by Lemma ~\ref{lem:check-sings-cox} to show that $(\overline{X},\overline{B})$ is log canonical. Since the boundary divisor $\overline{B}$ is $H_X$-invariant, it suffices to show that the pair $(\overline{X},\overline{B})$ is log canonical in a neighborhood of the vertex. By inversion of adjunction, it suffices to show that the pair $$\left(Y=\Spec \kk[x,y,z]/(x^2-y^2z+4z^3),B_Y=\frac{1}{2}\{z=0\}\right)$$ is log canonical. We have a quasi-\'etale morphism $$Z=\Spec \kk[\alpha,\beta,\gamma]/(\alpha\beta-\gamma^4)\rightarrow Y$$ given by $$x\mapsto (\alpha-\beta)\gamma,$$
		$$y\mapsto \alpha+\beta,$$
		$$z\mapsto \gamma^2.$$
		Under this map, the $\qq$-divisor $\frac{1}{2}\{z=0\}$ pulls back to the divisor $\{\gamma=0\}.$ Since the pair $(Z,\{\gamma=0\})$ is log canonical, it follows that $(Y,B_Y)$ is as well.
	\end{proof}
	
	Our construction above only works for dimension $\geq 3.$ It is interesting to know if the same can be achieved in dimension 2.
	
	\begin{question}
		Is there a surface $X$ with $\widehat{c}(X) = \frac{3}{2}$? Does the possible values of $\widehat{c}(X)$ depend on the dimension $n$?
	\end{question}
	
	Finally, we prove that if $X$ is of Fano type, then the absolute complexity is achieved as the complexity of a log Calabi--Yau pair $(X,B)$.

	\begin{proof}[Proof of Theorem ~\ref{introthm:abs-comp-is-min}]
		It follows from Corollary ~\ref{cor:compl-dlt-mod} and Lemma ~\ref{lem:abs-comp-bir-contr} that it suffices to prove the result for a small $\qq$-factorial model of $X.$ Thus, we may assume $X$ is $\qq$-factorial. Since $X$ is a $\qq$-factorial variety of Fano type, we may fix an integer $m>0$ such that $mD$ is Cartier for any prime divisor $D$ on $X.$\par
		Let $\{(X,B_n)\}_{n=1}^\infty$ be a sequence of log Calabi--Yau pairs on $X$ such that $c(X,B_n)$ decreases to $\widehat{c}(X).$ For each $n\geq 1,$ write $$B_n=\sum_{i=1}^{r_n}a_{n,i}D_{n,i}+\sum_{j=1}^{s_n}b_{n,j}E_{n,j}+\sum_{k=1}^{t_n}d_{n,k}F_{n,k}$$
		so that the following conditions are satisfied:
		\begin{enumerate}
			\item[(i)] each $D_{n,i}$ and $E_{n,j}$ is a movable prime divisor,
			\item[(ii)] each $D_{n,i}$ is $(m+1)$-very movable,
			\item[(iii)] no $E_{n,j}$ is $(m+1)$-very movable,
			\item[(iv)] each $F_{n,k}$ is prime divisor which is not movable.
		\end{enumerate}
		We claim that $\lim_{n\rightarrow \infty}\sum_{i=1}^{r_n}a_{n,i}=0.$ To show this, we will assume that there is some $\epsilon>0$ such that $\sum_{i=1}^{r_n}a_{n,i}>\epsilon$ for infinitely many $n\geq 1$ and will obtain a contradiction. Let $\{n_p\}_{p=1}^\infty$ be those positive integers with $\sum_{i=1}^{r_{n_p}}a_{n_p,i}>\epsilon$. Fix an arbitrary $p\geq 1.$ For $1\leq i \leq r_{n_p},$ let $X'$ be a small $\qq$-factorial model of $X$ on which the strict transform $D'_{n_p,i}$ of $D_{n_p,i}$ is nef and $mD'_{n_p,i}\sim \sum_{l=1}^{m+1}M_l$ for basepoint-free Cartier divisors $M_1,\hdots, M_{m+1}.$ Let $D'_l\in |M_l|$ be general for each $1\leq l \leq m+1.$ Then $\sum_{l=1}^{m+1}D_l$ is a general member of $|mD'_{n_p,i}|.$ Denoting by $B'_{n_p}$ the strict transform on $X'$ of $B_{n_p}$, we have that $(X',B'_{n_p})$ is log Calabi--Yau. Since $a_{n_p,i}D'_{n_p,i}=\frac{a_{n_p,i}}{m}(mD'_{n_p,i})$ is a summand of $B'_{n_p},$ it follows that $(X',B'_{n_p}+\frac{a_{n_p,i}}{m}\sum_{l=1}^{m+1}M_l-a_{n_p,i}D'_{n_p,i})$ is a log Calabi--Yau pair on $X'.$ Replace $B_{n_p}$ with the strict transform of $B'_{n_p}+\frac{a_{n_p,i}}{m}\sum_{l=1}^{m+1}M_l-a_{n_p,i}D'_{n_p,i}$ and repeat this process for each $1\leq i \leq r_{n_p}.$ We obtain a new log Calabi--Yau pair $(X,B''_{n_p})$ with $|B''_{n_p}|=|B_{n_p}|+\frac{1}{m}\sum_{i=1}^{r_{n_p}}a_{n_p,i}>|B_{n_p}|+\frac{\epsilon}{m},$ hence with $c(X,B''_{n_p})<c(X,B_{n_p})-\frac{\epsilon}{m}.$ But since $\lim_{p\rightarrow \infty}c(X,B_{n_p})=\widehat{c}(X),$ it follows that $c(X,B''_{n_p})<\widehat{c}(X)$ for all $p>>0.$ This is nonsense.\par
		For each $n\geq 1,$ write $\widetilde{B}_n=\sum_{j=1}^{s_n}b_{n,j}E_{n,j}+\sum_{k=1}^{t_n}d_{n,k}F_{n,k}.$ It follows from the paragraph above that $\lim_{n\rightarrow \infty}c(X,\widetilde{B}_n)=\widehat{c}(X).$ Since $0\leq \widetilde{B}_n\leq B_n,$ it follows that $(X,\widetilde{B}_n)$ is a log canonical pair for each $n\geq 1.$ We also note that $-(K_X+\widetilde{B}_n)\equiv\sum_{i=1}^{r_n}a_{n,i}D_{n,i}$ is movable for each $n\geq 1.$ It follows from Lemma ~\ref{lem:fin-m-mov} that there are finitely many prime divisors $G_1,\hdots, G_q$ on $X$ such that $\{E_{n,1},\hdots, E_{n,s_n},F_{n,1},\hdots, F_{n,t_n}\}\subset \{G_1,\hdots G_q\}$ for each $n\geq 1.$ For each $n\geq 1,$ write $$\widetilde{B}_n=\sum_{l=1}^qe_{n,l}G_l.$$ Since $e_{n,l}\in [0,1]$ for each $n\geq 1$ and $1\leq l \leq q,$ we may choose a subsequence $\{n_p\}_{p=1}^\infty$ such that $e_l:=\lim_{p\rightarrow \infty}e_{n_p,l}$ exists for each $1\leq l \leq q.$ Define $$\widetilde{B}=\sum_{l=1}^qe_lG_l.$$
		Then $(X,\widetilde{B})$ is log canonical, $-(K_X+\widetilde{B})$ is movable and $c(X,\widetilde{B})=\widehat{c}(X).$ Since $X$ is of Fano type, there is an small $\qq$-factorial modification $Y$ on which $-(K_Y+\widetilde{B}_Y)$ is nef. Here, $\widetilde{B}_Y$ is the strict transform on $Y$ of $\widetilde{B}.$ We claim that $(Y,\widetilde{B}_Y)$ is log canonical. To see this, let $B_{Y,n}$ and $\widetilde{B}_{Y,n}$ be the strict transfroms on $Y$ of $B_n$ and $\widetilde{B}_n$ for each $n.$ Each pair $(Y,B_{Y,n})$ is a log Calabi--Yau pair. Since $0\leq \widetilde{B}_{Y,n}\leq B_{Y,n}$ for each $n,$ it follows that each pair $(Y,\widetilde{B}_{Y,n})$ is log canonical. As for $(X,\widetilde{B})$, it follows that $(Y,\widetilde{B}_Y)$ is log canonical. Since $Y$ is of Fano type, it follows that the nef divisor $-(K_Y+\widetilde{B})$ is a linear combination of semi-ample divisors. It follows that there is a $B'_Y\geq \widetilde{B}_Y$ such that $(Y,B'_Y)$ is log Calabi--Yau. Denoting by $B'$ the strict transform on $X$ of $B'_Y,$ we obtain $B'\geq \widetilde{B}$ such that $(X,B')$ is log Calabi--Yau. But then $$\widehat{c}(X)\leq c(X,B')\leq c(X,\widetilde{B'})=\widehat{c}(X),$$
		so it follows that $c(X,B')=\widehat{c}(X).$\par
		Let $V=\bigoplus_{i=1}\rr B_i$, where $B_1,\hdots, B_r$ are the irreducible components of the support of $B'.$ It follows from ~\cite[1.3.2]{Sho93b} that the subset $\mathcal{L}=\{\Delta \in V\mid (X,\Delta)\text{ is log canonical}\}$ is a rational polytope. The set $\mathcal{C}=\{\Delta \in V\mid (X,\Delta)\text{ is log Calabi--Yau}\}$ is the intersection of $\mathcal{L}$ with the rational affine subspace that is the the preimage of $-K_X$ under the map $V\rightarrow N^1(X)$ sending a divisor to its numerical class. Thus, $\mathcal{C}$ is also a rational polytope. It follows that the affine function $$\widehat{c}\colon \mathcal{C}\rightarrow \rr$$ must attain its minimum on some vertex of $\mathcal{C}.$ We replace $B'$ by the divisor $B$ corresponding to one such vertex. We see that $B$ is a $\qq$-divisor, and hence that $\widehat{c}(X)=c(X,B)\in \qq.$
	\end{proof}
	
	\section{Complexity One and Cluster Type}
	
	In this section we prove Theorems \ref{introthm:fine-compl-one-cluster-type} and \ref{introthm:compl-one-quot-of-cluster-type}, which relates complexity one to cluster type varieties. First, we prove Theorem \ref{introthm:fine-compl-one-cluster-type}. Our proof is inspired by the proof of ~\cite[Theorem 1.2]{Mor24}.
	
	\begin{proof}[Proof of Theorem~\ref{introthm:fine-compl-one-cluster-type}]
		Write $n=\dim X.$ To show that $(X,B)$ is cluster type, it suffices by Lemma ~\ref{lem:check-cluster-type-on-bir-model} to show that $(X_0,B_0)$ is cluster type for any log Calabi--Yau pair (necessarily of index one) admitting a crepant birational map $(X_0,B_0)\dashrightarrow (X,B)$ that extracts only log canonical places of $(X,B).$ It follows from ~\cite[Theorem 5.6]{mor_con24} and its proof that there exists a diagram 
		\[
		\xymatrix{
			(X,B) & (X_0,B_0) \ar@{-->}[l]_-{\phi} \ar[r]^-{\psi_0} & (X_1,B_1+\mathbf{M}_1)\ar[r]^-{\psi_1} & \hdots \ar[r]^-{\psi_{r-1}} & (X_r,B_r+\mathbf{M}_r) \ar[r] & V
		}
		\]
		and a subset $I\subset \{0,\hdots, r-1\}$ satisfying the following conditions:
		\begin{enumerate}
			\item[(i)] $\phi$ is crepant birational and extracts only log canonical places of $(X,B),$
			\item[(ii)] each $\psi_i$ is a Mori fiber space,
			\item[(iii)] each $X_i$ is of Fano type over $V,$
			\item[(iv)] $|I|\geq n-1,$
			\item[(v)] each $\psi_i$ with $i\in I$ is a conic fibration,
			\item[(vi)] if $i\in I,$ then $\lvert\lfloor B_{i,{\rm horiz}} \rfloor\rvert\geq 2.$
		\end{enumerate}
		Here, the $(X_i,B_i+\mathbf{M}_i)$ are the generalized pairs induced by the canonical bundle formula. Since each $\psi_i$ is a fibration and $|I|\geq n-1,$ it follows that each $\psi_i$ must have relative dimension one. Thus, we have $\dim X_{i}=n-i$ for each $0\leq i \leq r.$ In particular, $X_{n-1}$ is a rationally connected curve and hence must be isomorphic to $\mathbb{P}^1.$ If $I=\{0,\hdots, n-1\},$ we set $j_1=n$ and $j_2=-1.$ Otherwise, we set $j_1=j_2$ to be the unique index in $\{0,\hdots, n-1\}\setminus I.$ In the event that $r=n-1,$ we set $X_n=\Spec \kk$, $B_n=0$ and $\mathbf{M}_n=0$.\par
		We note that the hypotheses of Lemma ~\ref{lem:q-fact-tower} are satisfied. Thus, we may assume that:
		\begin{enumerate}
			\item[(vii)] each $X_i$ is $\qq$-factorial,
			\item[(viii)] $\psi_i^{-1}(\lfloor B_{i+1}\rfloor)\subset \lfloor B_i\rfloor$ for each $0\leq i \leq r-1.$
		\end{enumerate}
		Since each $\psi_i$ is of relative dimension one, it follows that no horizontal divisors can be contracted upon application of Lemma ~\ref{lem:q-fact-tower}. Thus, condition (vi) is preserved.
		\par

		We will show that $(X,B)$ is cluster type by showing that $(X_0,B_0)$ is a toric log Calabi--Yau pair. We begin by claiming that all moduli b-divisors $\mathbf{M}_i$ are trivial as b-divisors. We first show, for each $0\leq i\leq \min\{j_1,n-1\}$, that $\mathbf{M}_i$ is trivial as b-divisor and that $(X_i,B_i)$ is a log Calabi--Yau pair of index one. The case $i=0$ is clear, and the case $i$ implies the case $i+1$ by Lemma ~\ref{lem:no-moduli-gen-finite-lcc} and the fact that conditions (v) and (vi) imply that the components of $B_{i-1,{\rm horiz}}$ map birationally onto $X_i.$ If $j_1=n,$ then we have proven the claim. Otherwise, it follows from Lemma ~\ref{lem:conic-fib-horiz-divs} and the fact that $(X_{j_1},B_{j_1})$ is a log Calabi--Yau pair of index one that we must have $|\lfloor B_{j_1,{\rm horiz}}\rfloor|=1$. Thus, we may conclude from conditions (v) and (vi) and further applications of Lemma ~\ref{lem:no-moduli-gen-finite-lcc} that the moduli b-divisors $\mathbf{M}_i$ for $j_1<i\leq n-1$ are also trivial as b-divisors. \par
		We claim now that $B_i=\lfloor B_i\rfloor$ and $c(X_i,B_i)=0$ for each $j_2<i\leq n-1.$ We start with the case $i=n-1.$ If $j_1\neq n-1,$ then this follows from the fact that $X_{n-1}\cong \pp^1$ and that $|\lfloor B_{n-1}\rfloor|\geq 2$ by condition (vi). If $j_1=n-1,$ then this follows from the fact that $X_{n-1}\cong\pp^1$ and that $(X_{n-1},B_{n-1})$ is a log Calabi--Yau pair of index one. The case $i$ implies the case $i-1$ by noting that  
		$$c(X_{i-1},\lfloor B_{i-1}\rfloor)\leq (\dim X_i+1) + (\rho(X_i)+1)-(|B_i|+2) \leq 0.$$
		Indeed, we have that $\rho(X_{i-1})=\rho(X_i)+1$ by condition (ii), that $\dim X_{i-1}=\dim X_i+1$ by condition (v), that $|\lfloor B_{i-1,{\rm horiz}}\rfloor |\geq 2$ by condition (vi), and that $|\lfloor B_{i-1,{\rm vert}}\rfloor |\geq |B_i|$ by condition (viii). Therefore, we must have $B_i=\lfloor B_{i-1}\rfloor$ and $c(X_{i-1},B_{i-1})=0$ by Theorem ~\ref{introthm:BMSZ18}. \par
		In the event that $j_2=-1,$ we have just shown that $c(X_0,B_0)=0.$ So assume that $j_2\neq -1$, hence that $j_2=j_1.$ We showed above that $(X_{j_2},B_{j_2})$ is a log Calabi--Yau pair of index one, so it follows from Lemma ~\ref{lem:conic-fib-horiz-divs} that we must have $|\lfloor B_{j_2,{\rm horiz}}\rfloor|\geq 1$. The arguments of the previous paragraph show that $c(X_{j_2},B_{j_2})\in \{0,1\}.$ We cannot have $c(X_{j_2},B_{j_2})=1$ by Lemma ~\ref{lem:compl-one-horiz-divs}, so we must have $c(X_{j_2},B_{j_2})=0.$ Further applications of the arguments of the previous paragraph now show that $c(X_i,B_i)=0$ for each $0\leq i < j_2.$ In particular, $(X_0,B_0)$ is a toric log Calabi--Yau pair by Theorem \ref{introthm:BMSZ18}.
		
	\end{proof}
	
	We obtain Theorem \ref{introthm:compl-one-quot-of-cluster-type} as a consequence of Theorem \ref{introthm:abs-compl-cA} and Theorem \ref{introthm:fine-compl-one-cluster-type}.
	
	\begin{proof}[Proof of Theorem~\ref{introthm:compl-one-quot-of-cluster-type}]
		It suffices to show the result when $X$ is $\qq$-factorial. Indeed, if $f\colon X'\rightarrow X$ is a small birational morphism and $g'\colon Y'\rightarrow X'$ is a finite morphism, then Stein factorization of the composite $f\circ g'$ gives us a commutative diagram 
		\[
		\xymatrix{ 
			Y'\ar[d]_-{g'} \ar[r]^-{h} & Y\ar[d]^-{g} \\ 
			X' \ar[r]^-{f} & X
		}
		\]
		with $h$ a small birational morphism and $g$ a finite morphism of degree equal to that of $g'.$ In particular, $Y$ will be cluster type if and only if $Y'$ is.\par
		By Theorem ~\ref{introthm:abs-compl-cA}, $X$ is a Fano type variety and $\overline{X}$ has a cA-type singularity at its vertex. As in the proof of Theorem ~\ref{introthm:abs-compl-cA}, we have a commutative diagram  
		\[
		\xymatrix{ 
			\overline{X} \ar@{^{(}->}[r]^-{\overline{j}}& \mathbb{A}^{d+1}=\overline{T}\\ 
			\widehat{X}\ar[u]^-{i_X} \ar[d]_-{p_X} \ar@{^{(}->}[r]^-{\widehat{j}} & \widehat{T}\ar[u]_-{i_T}\ar[d]^-{p_T} \\ 
			X \ar@{^{(}->}[r]^-{j} &
			T
		}
		\]
		satisfying the following conditions:
		\begin{enumerate}
			\item[(i)] the morphism $\overline{j}$ is the inclusion of a hypersurface defined by a ${\rm Cl}(X)$-homogeneous equation $f\in \kk[x_1,\hdots, x_{d+1}]={\rm Cox}(T)$ as in Lemma ~\ref{lem:cA-cox},
			\item[(ii)] $T$ is a projective toric variety on which $X={\rm div}_{T}(f).$ 
		\end{enumerate}
		Write $D_i={\rm div}_T(x_i)$ for $1\leq i \leq d+1.$ Thus, $\sum_{i=1}^{d+1}D_i$ is the toric boundary on $T.$ It was shown in the proof of Theorem ~\ref{introthm:abs-compl-cA} that $(T, X+\sum_{i=3}^{d+1}D_i)$ is log canonical.
		
		First, suppose that $f$ contains the monomial $x_1x_2,$ then $X\sim D_1+D_2$ and it follows that $(T, X+\sum_{i=3}^{d+1}D_i)$ is a log Calabi--Yau pair of index one. Denoting by $(X,B)$ the pair obtained by adjunction to $X,$ it follows that $(X,B)$ is a log Calabi--Yau pair of index one with $c(X,B)=1.$ Since $X$ is of Fano type, it is rationally connected by ~\cite[Theorem 1]{Zha06}. Thus, it follows from Theorem ~\ref{introthm:fine-compl-one-cluster-type} that $(X,B)$ is cluster type, hence that $X$ is cluster type.\par
		
		Now, assume that $f$ does not contain the monomial $x_1x_2.$ Thus, $f$ contains the monomial $x_1^2+x_2^2$ and $E={\rm deg}(x_1)-{\rm deg}(x_2)$ is a nonzero $2$-torsion element of ${\rm Cl}(X)={\rm Cl}(T).$ Let $H'=\Spec \kk[{\rm Cl}(X)/ E]$, and denote by $S=\widehat{T}/H'$ and $Y=\widehat{X}/H'.$ We obtain a commutative diagram 
		\[
		\xymatrix{ 
			Y \ar[d]_-{\pi_X} \ar[r] & S\ar[d]^-{\pi_T} \\ 
			X \ar@{^{(}->}[r]^-{j} &
			T,
		}
		\]
		and the following conditions hold:
		\begin{enumerate}
			\item[(i')] $S$ is a $\qq$-factorial projective toric variety with ${\rm Cl}(S)={\rm Cl}(T)/E,$ and ${\rm Cox}(S)={\rm Cox}(T)$ with ${\rm Cl}(S)$-grading obtained by coarsening the ${\rm Cl}(T)$-grading, and
			\item[(ii')] $Y$ is the normalization of ${\rm div}_S(f)\subset S,$
			\item[(iii')] $\pi_T$ and $\pi_X$ are quasi-\'etale morphisms of degree $2.$ 
		\end{enumerate}
		Denoting by $D_{S,i}={\rm div}_S(x_i)$ for $1\leq i \leq d+1,$ it follows as in the proof of Theorem ~\ref{introthm:abs-compl-cA} that $(S,Y+\sum_{i=3}^{d+1}D_{S,i})$ is log canonical. Since $E={\rm deg}(x_1)-{\rm deg}(x_2)$ is zero in ${\rm Cl}(S),$ it follows that $S\sim D_1+D_2$ and hence that $(S, Y+\sum_{i=3}^{d+1}D_{S,i})$ is a log Calabi--Yau pair of index one. We obtain a log Calabi--Yau pair $(Y,B_Y)$ of index one by performing adjunction to $Y.$ We note that the finite morphism $\pi_X$ is crepant with respect to the pairs $(Y,B_Y)$ and $(X,B).$ In particular, it follows from ~\cite[Lemma 2.17]{mor_con24} and the fact that $X$ is of Fano type that $Y$ is of Fano type as well. Since $(X,B)$ is a log Calabi--Yau pair with $c(X,B)=1,$ it follows from Theorem ~\ref{introthm:compl-less-than-2} that the components of $B$ span a subspace of $N^1(X)$ of dimension $|B|+1-\dim X.$ Since $B_Y=f^*B_X$, it follows that $\overline{c}(Y,B_Y)=1.$ Since $Y$ is of Fano type, it is rationally connected by ~\cite[Theorem 1]{Zha06}. Thus, it follows from Theorem ~\ref{introthm:fine-compl-one-cluster-type} that $(Y,B_Y)$ is cluster type, hence that $Y$ is cluster type.
	\end{proof}
	
	\begin{remark}
		{\em 
			The example of ~\cite[Section 7]{BMSZ18} shows that it is necessary in general to allow for double covers.
		}
	\end{remark}
	
	\section{\texorpdfstring{$T$}{T}-varieties and Cluster Type}
	
	In this final section, we prove Theorem \ref{introthm:t-compl-one}. We refer the reader to \cite{AIPSV12} and references therein for definitions and results about $T$-varieties.
	
	\begin{proof}[Proof of Theorem ~\ref{introthm:t-compl-one}]
		By ~\cite[Theorem 5.6]{AHS08}, the variety $X$ arises from a \textit{divisorial fan} $\mathcal{S}=\{\mathcal{D}^i\mid i\in I\}$ on a smooth projective curve $C$ as in ~\cite[Definition 5.2]{AHS08}. For each $i\in I,$ we have a diagram 
		\[
		\xymatrix{ 
			\widetilde{X}(\mathcal{D}^i) \ar[d]_-{\pi_i} \ar[r]^-{r_i} & X(\mathcal{D}^i) \\ 
			C  &
		}
		\]
		in which the morphism $r_i$ is projective, birational and $\mathbb{T}$-equivariant and the morphism $\pi_i$ is $\mathbb{T}$-invariant. These varieties and morphisms glue together to give a diagram 
		\[
		\xymatrix{ 
			\widetilde{X} \ar[d]_-{\pi} \ar[r]^-{r} & X \\ 
			C  &
		}
		\]
		in which the morphism $r$ is projective, birational and $\mathbb{T}$-equivariant and the morphism $\pi$ is $\mathbb{T}$-invariant. Moreover, there is a nonempty open $U\subset C$, a toric variety $F$ with acting torus $\mathbb{T}$ and a $\mathbb{T}$-equivariant isomorphism $\pi^{-1}(U)\cong U\times F$ over $U.$\par 
		As in the proof of ~\cite[Theorem 3.25]{Mor22}, we may choose a $\mathbb{T}$-invariant complement $(X,B).$ Let $(\widetilde{X},\widetilde{B})$ denote the log pullback of $(X,B)$ via $r.$ Performing adjunction to a general fiber $F$ of $\pi,$ we obtain a $\mathbb{T}$-invariant log Calabi--Yau sub-pair $(F,B_F)$ on $F.$ It follows that $B_F$ must be the toric boundary of $F.$ In particular, every component of $B_F$ has coefficient one. It follows from ~\cite[Proposition 3.13]{PS11} that every horizontal $\mathbb{T}$-invariant divisor on $\widetilde{X}$ appears in $\widetilde{B}$ with coefficient one. It follows from ~\cite[Theorem 10.1]{AH06} that $r$ contracts no divisors on $\widetilde{X}$ that are vertical over $C.$ In other words, every $r$-exceptional divisor appears in $\widetilde{B}$ with coefficient one. It follows that $\widetilde{B}$ is effective, and by ~\cite[Lemma 2.23]{Mor21} that $\widetilde{X}$ is Fano type. It now follows from ~\cite[Theorem 5.2]{FG12} that $C=\pp^1.$\par
		Since $\widetilde{X}$ is Fano type, and therefore log terminal, by taking a $T$-equivariant small $\qq$-factorialization, we may assume that $\widetilde{X}$ is $\qq$-factorial. The pair $(\widetilde{X},\widetilde{B}_{\rm horiz})$ is log Calabi--Yau over $C$. Denote by $(C,\widetilde{B}_C)$ the pair induced from $(\widetilde{X},\widetilde{B}_{\rm horiz})$ on $C$ via the canonical bundle formula. We note that the moduli part is trivial here by Lemma ~\ref{lem:no-moduli-gen-finite-lcc} since the pair $(\widetilde{X},\widetilde{B}_{\rm horiz})$ has one-dimensional log canonical centers that are horizontal over $C.$ The pair $(C,\widetilde{B}_C)$ has standard coefficients by ~\cite[Remark 3.1.4]{Amb99}, hence for each $p\in C$ there is an integer $m_p\geq 1$ such that ${\rm coeff}_{p}(\widetilde{B}_C)=1-\frac{1}{m_p}.$ We claim that for each $p\in C,$ there exists a component of the scheme-theoretic fiber $\pi^{-1}(p)$ with multiplicity $m_p.$ Denote by $m'_p$ the maximal multiplicity of a fiber in $\pi^{-1}(p).$ To show the claim, it will suffice to show that $(\widetilde{X},\widetilde{B}_{\rm horiz}+\frac{1}{m'_p}\pi^{-1}(p))$ is log canonical. By definition of $m'_p,$ this will follow if we show that $(\widetilde{X},\widetilde{B}_{\rm horiz}+\pi^{-1}(p)_{\rm red})$ is log canonical. Since $C=\pp^1,$ it follows from ~\cite[Example 2.5]{LS13} that there is an affine open neighborhood of $p$ in $C$ over which the pair $(\widetilde{X},\widetilde{B}_{\rm horiz}+\pi^{-1}(p)_{\rm red})$ is isomorphic to the restriction of a toric pair. This implies the desired log canonicity.\par
		We claim that $K_C+\widetilde{B}_C$ is not numerically trivial. To show this claim, we will assume that it is numerically trivial and will obtain a contradiction. Since $\widetilde{X}$ is a Fano type variety, it follows from the previous paragraph that we may run a sequence of relative MMPs over $C$ to obtain a commutative diagram 
		\[
		\xymatrix{  
			\widetilde{X} \ar[rd]_-{\pi} \ar@{-->}[rr]^-{f} & & X' \ar[ld]^-{\pi'}\\
			& C &
		}
		\]
		in which $f$ is a contracting birational map and such that, for each $p\in C$, $\pi'^{-1}(p)$ is irreducible with multiplicity $m_p.$ These MMPs are automatically $\mathbb{T}$-equivariant. In particular, $X'$ carries an effective $\mathbb{T}$-action. Consider the index one cover $\psi\colon E\rightarrow C$ of $K_C+\widetilde{B}_C.$ The curve $E$ has genus one and, for each $p\in C,$ the fiber $\psi^{-1}(p)$ is a disjoint union of points of multiplicity $m_p$. We note that $X'\times_C E$ is irreducible since it is flat over $E$ and is isomorphic to $\psi^{-1}(U)\times F$ over $\psi^{-1}(U)$ for some nonempty open $U\subset C.$ Let $Y$ denote the normalization of $X'\times_C E$ with its reduced structure. We have a commutative diagram
		\[
		\xymatrix{ 
			X' \ar[d]_-{\pi'} & Y \ar[d]^-{\pi''} \ar[l]_-{\phi} \\ 
			C & E \ar[l]_-{\psi}
		}
		\]
		with $\phi$ finite and surjective and $\pi''$ surjective. For each $q\in E,$ we have that the ramification index of $\psi$ at $q$ is equal to the multiplicity $m_{\psi(p)
		}.$ From this, it follows that the morphism $\phi$ is quasi-\'etale. In particular, it follows that $Y$ is Fano type. This implies, by ~\cite[Theorem 5.2]{FG12}, that $E$ must be Fano type. But $E$ has genus one, so this is nonsense. Having obtained the desired contradiction, we conclude that $K_C+\widetilde{B}_C$ is not numerically trivial.\par
		The pair $(C,\widetilde{B}_C)$ is a Fano pair on $C=\pp^1$ with standard coefficients. We first consider the case that $\Supp(\widetilde{B}_C)$ contains at most two points. Thus, there exist distinct points $p,q\in C$ with $\widetilde{B}_C\leq \widehat{B}_C:=p+q.$ Writing $\widehat{B}=\widetilde{B}_{\rm horiz}+\pi^*(\widehat{B}_C-\widetilde{B}_C)$ and $\widehat{B}'=f_*\widehat{B}=f_*\widetilde{B}_{\rm horiz}+\pi'^*(\widehat{B}_C-\widetilde{B}_C)$, it follows that $(\widetilde{X},\widehat{B})$ and $(X',\widehat{B}')$ are log Calabi--Yau pairs with respect to which $f$ is crepant. Note that every component of $\widehat{B}'$ has coefficient one. Since every fiber of $\pi'$ is irreducible and is isomorphic over some noempty open $U\subset C=\pp^1$ to the projection $U\times F\rightarrow U$, it follows that $\rho(X')=\rho(F)+1.$ Since the horizontal components of $\widehat{B}'$ are in bijection with the components of the toric boudnary of $F$, it follows that $|\widehat{B}'|=|B_F|+2=\dim F + \rho(F) + 2 = \dim X' + \rho(X').$ Thus, $(X',\widehat{B}')$ is a toric log Calabi--Yau pair by Theorem ~\ref{introthm:BMSZ18}. Since $f\colon \widetilde{X}\dashrightarrow X'$ is a contracting birational map, it follows that $(\widetilde{X},\widehat{B})$ is cluster type. Since $r\colon \widetilde{X}\rightarrow X$ contracts only divisors that are horizontal over $C$, all of which appear in $\widehat{B}$ with coefficient one, it follows that $(X,r_*\widehat{B})$ is cluster type. To prove the theorem in this case, we may take $Y=X$ and $B_Y=r_*\widehat{B}.$\par
		From now on, we assume that $\Supp(\widetilde{B}_C)$ contains more than two points. Since $(C,\widetilde{B}_C)$ is a Fano pair on $C=\pp^1$ with standard coefficients, we are left with the following possibilities:
		\begin{enumerate}
			\item[$(D_n)$] $\widetilde{B}_C=\frac{1}{2}p+\frac{1}{2}q+\frac{n-1}{n}r,$
			\item[$(E_6)$] $\widetilde{B}_C=\frac{1}{2}p+\frac{2}{3}q+\frac{2}{3}r,$
			\item[$(E_7)$] $\widetilde{B}_C=\frac{1}{2}p+\frac{2}{3}q+\frac{3}{4}r,$
			\item[$(E_8)$] $\widetilde{B}_C=\frac{1}{2}p+\frac{2}{3}q+\frac{4}{5}r.$
		\end{enumerate}
		Here, $p,q,r\in C$ are distinct points. In the case of $(D_n),$ consider a degree $2$ morphism $\psi\colon C'=\pp^1\rightarrow C$ branched at $p$ and $q.$ Denoting by $r_1,r_2\in C'$ the preimages of $r,$ we see that $\psi$ is crepant with respect to $(C',\widetilde{B}_{C'}=\frac{n-1}{n}(r_1+r_2))$ and $(C,\widetilde{B}_C).$ In the cases $E_n,$ we let $\gamma\colon C'=\pp^1\rightarrow C$ be the universal cover of the smooth orbifold $(C,\widetilde{B}_C)$ as in ~\cite[Corollary 1.1]{Clau08}. In this case, the fundamental group $\pi_1(C/\widetilde{B}_C)$ of the smooth orbifold $(C,\widetilde{B}_C)$ as in ~\cite[Definition 1.2]{Clau08} has presentation $\langle \gamma_1,\gamma_2\mid \gamma_1^2=\gamma_2^3=(\gamma_1\gamma_2)^{n-3}=1\rangle$. Here, $\gamma_1$ is a small loop $p$, $\gamma_2$ is a small loop around $q,$ and $\gamma_1\gamma_2$ is homotopic in $C\setminus\{p,q,r\}$ to a loop around $r.$ Using these presentations to identify $\pi_1(C/\widetilde{B}_C)$ with $A_4,$ $S_4$, $A_5$ when $n=6,7,8$, respectively, we see that ${\rm ord}(\gamma_1)=2,$ ${\rm ord}(\gamma_2)=3$ and ${\rm ord}(\gamma_1\gamma_2)=n-3.$ It follows from ~\cite[Proposition 1.1]{Clau08} that $\psi$ has ramification indices of $2,$ $3,$ $n-3$ over $p,$ $q$, $r$, respectively, and is \'etale elsewhere. In particular, $\psi$ is crepant with respect to the pairs $(C',\widetilde{B}_{C'}=0)$ and $(C,\widetilde{B}_C).$ \par
		As above, we note that $X'\times_C C'$ is irreducible. Let $Y'$ denote the normalization of $X'\times_C C'$ with its reduced structure. We have a commutative diagram
		\[
		\xymatrix{ 
			X' \ar[d]_-{\pi'} & Y' \ar[d]^-{\pi_{Y'}} \ar[l]_-{\phi'} \\ 
			C & C' \ar[l]_-{\psi}.
		}
		\]
		For each $t\in C',$ we have seen that the ramification index of $\psi$ at $t$ is equal to the multiplicity of the fiber $\pi'^{-1}(\psi(t)).$ It follows that the morphism $\phi$ is quasi-\'etale. Let $\widetilde{Y}$ denote the normalization of $\widetilde{X}\times_CC'$ with its induced structure. As above, $\widetilde{Y}$ is irreducible. Moreover, the induced map $h\colon \widetilde{Y}\dashrightarrow Y'$ is a contracting birational map. Denote by $\widetilde{\phi}\colon \widetilde{Y}\rightarrow \widetilde{X}$ the induced morphism and by $$\widetilde{Y}\xrightarrow{s}Y\xrightarrow{\phi}X$$ the Stein factorization of $\widetilde{Y}\xrightarrow{\widetilde{\phi}}\widetilde{X}\xrightarrow{r}X.$ We obtain a commutaive diagram 
		\[
		\xymatrix{ 
			X & Y \ar[l]_-{\phi} \\ 
			\widetilde{X} \ar@/_2pc/[dd]_-{\pi} \ar@{-->}[d]_-{f} \ar[u]^-{r} & \widetilde{Y} \ar@/^2pc/[dd]^-{\pi_{\widetilde{Y}}} \ar@{-->}[d]^-{h} \ar[u]_-{s} \ar[l]_-{\widetilde{\phi}} \\ 
			X' \ar[d]_-{\pi'} & Y' \ar[d]^-{\pi_{Y'}} \ar[l]_-{\phi'} \\ 
			C & C' \ar[l]_-{\psi}.
		}
		\]
		Since both $\pi$ and $\pi'$ are $\mathbb{T}$-invariant morphisms, both $\widetilde{Y}$ and $Y'$ effective inherit $\mathbb{T}$-actions in such a way that $\widetilde{\phi},$ $\phi',$ $h$ are $\mathbb{T}$-equivariant and $\pi_{\widetilde{Y}},$ $\pi_{Y'}$ are $\mathbb{T}$-invariant. Since $r\circ \widetilde{\phi}$ is $\mathbb{T}$-equivariant, it follows that there exists a $\mathbb{T}$-action on $Y$ such that the factors $s$ and $\phi$ of its Stein factorization are $\mathbb{T}$-equivariant.
		
		Since $\phi'$ is quasi-\'etale, it follows that $Y'$ is Fano type and that $\phi'$ is crepant with respect to the pairs $(X', B'_{\rm horiz}=f_*\widetilde{B}_{\rm horiz})$ and $(Y',B_{Y',{\rm horiz}}=\phi'^*B'_{\rm horiz})$. As above, we may assume that $Y'$ is $\qq$-factorial, and we may run a $\mathbb{T}$-equivariant MMP over $C'$ to obtain a commutative diagram 
		\[
		\xymatrix{  
			Y' \ar[rd]_-{\pi_{Y'}} \ar@{-->}[rr]^-{k} & & Y'' \ar[ld]^-{\pi_{Y''}}\\
			& C' &
		}
		\]
		in which $k$ is a contracting birational map and such that, for each $p\in C'$, $\pi_{Y''}^{-1}(p)$ is irreducible with multiplicity $\frac{1}{1-{\rm coeff}_p(\widetilde{B}_{C'})}.$ We have seen that the support of $B_{C'}$ consists of at most two points. Thus, there exist distinct points $p,q\in C'$ with $\widetilde{B}_C'\leq B_{C'}=p+q$. It follows that $\widehat{B}_{Y''}=B_{Y'',{\rm horiz}}+\pi_{Y''}^*B_{C'}$ is a reduced divisor on $Y''$ such that $(Y'', \widehat{B}_{Y''})$ is log Calabi--Yau. As above, we compute $|\widehat{B}_{Y''}|=\dim Y''+\rho(Y'')$, so it follows from Theorem ~\ref{introthm:BMSZ18} that $(Y'',\widehat{B}_{Y''})$ is a toric log Calabi--Yau pair. The acting torus can be identified with ${\rm Aut}^0(Y'',\widehat{B}_{Y''}),$ and we have an inclusion $\mathbb{T}\hookrightarrow {\rm Aut}^0(Y'',\widehat{B}_{Y''})$ since $\widehat{B}_{Y''}$ is $\mathbb{T}$-invariant. Denoting by $U''=Y''\setminus \widehat{B}_{Y''}$ the open orbit of ${\rm Aut}^0(Y'',\widehat{B}_{Y''}),$ it follows that $U''$ is a $\mathbb{T}$-invariant open on which $\mathbb{T}$ acts freely. The birational map $k\circ h\colon \widetilde{Y}\dashrightarrow Y''$ is $\mathbb{T}$-equivariant and is surjective in codimension one. Thus, there is a $\mathbb{T}$-invariant open $V''\subset Y''$ with complement of codimension at least two in $Y''$ on which $(k\circ h)^{-1}$ is defined as an open immersion $V''\hookrightarrow \widetilde{Y}.$ It follows that $V''\cap U''$ is a $\mathbb{T}$-invariant open whose complement has codimension at least two in $U''.$ Since $\dim \mathbb{T}=\dim U''-1$ and $\mathbb{T}$ acts freely on $U'',$ it follows that all $\mathbb{T}$-orbits in $U''$ are divisors in $U''.$ It follows that we must have $U''\subset V'',$ hence we have an open immersion $U''\hookrightarrow \widetilde{Y}$. Denoting by $(\widetilde{Y},\widehat{B}_{\widetilde{Y}})$ the log pullback of $(Y'', \widehat{B}_{Y''})$, we note that the image of this embedding lies in the complement of the support of $\widehat{B}_{\widetilde{Y}}.$ Notice that $s$ contracts only divisors that are horizontal over $C',$ and every such divisor has divisorial center on $Y''.$ It follows that $s$ contracts no divisor in $U''.$ Since $s$ is $\mathbb{T}$-equivariant, it follows as above that $s$ restricts to an open immersion on $U''$ with image contained in the complement of the support of $B_Y=s_*\widehat{B}_{\widetilde{Y}}.$ 
	\end{proof}

	\bibliographystyle{habbvr}
	\bibliography{references}
	
\end{document}